\documentclass[12pt]{amsart}

\title[Hyperbolic contact symplectic lifts]{Hyperbolic contact symplectic lifts}
\author[F. Bracci]{Filippo Bracci}
\author[B. McKay]{Benjamin McKay}
\author[R. Ugolini]{Riccardo Ugolini}

\usepackage[utf8]{inputenc}
\usepackage[T1]{fontenc}
\usepackage{times}
\usepackage{textcomp}
\usepackage{mathtext}
\usepackage[pdftex,pagebackref]{hyperref}
\hypersetup{
colorlinks   = true,  
urlcolor     = black, 
linkcolor    = black, 
citecolor    = black  
}
\usepackage[kerning=true,tracking=true]{microtype}
\usepackage{url}
\usepackage{amsmath}
\usepackage{amssymb}
\usepackage{amsthm}
\usepackage{mathrsfs}
\usepackage{mathtools}
\usepackage{varioref}
\usepackage{amssymb}
\usepackage{amsmath}
\usepackage{comment}
\usepackage{enumitem}
\usepackage{tcolorbox}

\usepackage{mathtools}
\DeclarePairedDelimiter{\SET}{\lbrace}{\rbrace}

\DeclareDocumentCommand{\set}{ o m }
  { \IfNoValueTF{#1}
    { \SET*{\,#2\,} }
    { \SET*{\,#1\,:\,#2\,} }
  }

\usepackage{tikz-cd}
\usetikzlibrary{
arrows.meta,
calc,
fit, 
positioning, 
shapes.geometric,
shapes.multipart
}
\numberwithin{equation}{section}

\NewDocumentCommand\ip{E{_}{{}}m}%
{%
\iota_{#1}{#2}%
}%

\newcommand{\dimR}{\dim_{\R}}
\newcommand{\dimC}{\dim_{\C}}
\renewcommand{\Re}{\operatorname{Re}}

\newcommand{\LieG}{\mathfrak{g}}

\newcommand{\LieL}{\mathfrak{l}}
\newcommand{\LieN}{\mathfrak{n}}
\DeclareMathOperator\Ad{Ad}
\newcommand{\R}{\mathbb R}
\newcommand{\C}{\mathbb C}
\newcommand{\B}{\mathbb B}
\DeclareMathOperator{\Aut}{Aut}
\newcommand{\D}{\mathbb D}
\newcommand{\N}{\mathbb N}
\newcommand{\LieDer}{{\mathcal L}}
\let\caron=\v
\DeclareMathOperator{\PSU}{PSU}
\DeclareMathOperator{\Arg}{Arg}
\def\Ha{\mathbb H}
\usepackage{newunicodechar}
\DeclareUnicodeCharacter{010D}{\caron{c}}
\newunicodechar{č}{\caron{c}}
\DeclareUnicodeCharacter{00E0}{\` a}

\NewDocumentCommand\ala{m}{\emph{à la} #1} 

\NewDocumentCommand\ie{}{i.e.\@ }
\NewDocumentCommand\eg{}{e.g.\@ }

\newtheorem{theorem}{Theorem}[section]
\newtheorem{lemma}[theorem]{Lemma}
\newtheorem{proposition}[theorem]{Proposition}
\newtheorem{corollary}[theorem]{Corollary}

\theoremstyle{definition}
\newtheorem{definition}[theorem]{Definition}
\newtheorem{example}[theorem]{Example}

\theoremstyle{remark}
\newtheorem{remark}[theorem]{Remark}
\numberwithin{equation}{section}

\makeatletter
\@namedef{subjclassname@2020}{\textup{2020} Mathematics Subject Classification}
\makeatother

\newlist{ProofStepList}{enumerate}{1}
\setlist[ProofStepList]{label=\arabic*. ,ref=\arabic*}

\NewDocumentEnvironment{ProofSteps}{}%
{%
\begin{ProofStepList}%
}%
{%
\end{ProofStepList}%
}%
\NewDocumentCommand\ProofStep{om}%
{%
\IfValueTF{#1}%
{
\item\label{#1}{\emph{#2}}\\
}
{
\item{\emph{#2}}\\
}
}%

\NewDocumentCommand\cls{m}{\ensuremath{[\![#1]\!]}}

\NewDocumentCommand\dotIndex{}{\scalebox{.5}{$\bullet$}}
\NewDocumentCommand\PSL{m}{\ensuremath{\mathbb{P}\!\operatorname{SL}_{#1}}}
\NewDocumentCommand\RP{m}{\ensuremath{\mathbb{RP}^{#1}}}

\date\today

\address{Filippo Bracci\newline
Dipartimento di Matematica e Informatica ``Ulisse Dini''\\
Università di Firenze\\
Viale Morgagni 67/a\\
50134 Firenze\\
Italy}
\email{filippo.bracci@unifi.it}
\thanks{
The first author is partially supported by GNSAGA of INdAM.
The third author was supported in part by the European Union (ERC Advanced grant HPDR, 101053085 to F.~Forstnerič).
This article is based upon work from COST Action CaLISTA CA21109 supported by COST (European Cooperation in Science and Technology).}

\address{Benjamin McKay\newline School of Mathematical Sciences,  University College Cork, Cork, Ireland}
\email{b.mckay@ucc.ie}

\address{Riccardo Ugolini\newline Department of Mathematics \\ 
Faculty of Mathematics and Physics \\
Jadranska 19 \\
1000 Ljubljana \\
Slovenia}
\email{riccardo.ugolini@fmf.uni-lj.si}

\subjclass[2020]{14J42, 32Q45, 53D35, 53E50}
\keywords{holomorphic contact forms; hyperbolicity of contact structure; holomorphic symplectic manifolds; Kobayashi hyperbolic}

\begin{document}
\begin{abstract}
Consider a holomorphic contact manifold.
Holomorphic discs tangent to the contact planes define a pseudometric on the manifold. This pseudometric integrates to a pseudodistance.
When the pseudodistance is a distance, we call the contact manifold \emph{contact-hyperbolic}, by analogy with Kobayashi hyperbolicity.

The goal of this paper is to construct explicit examples of contact-hyperbolic contact manifolds with large automorphism groups.

We study Reeb manifolds: holomorphic contact structures equipped with a Reeb vector field whose flow acts freely.
Our first main theorem shows that every proper Reeb manifold admits a holomorphic symplectic quotient.
It also identifies which symplectic manifolds arise this way.
The isomorphism classes of proper Reeb manifolds over a fixed symplectic base manifold are parameterised by the first cohomology.

Our second main theorem: a proper Reeb manifold is (complete) contact-hyperbolic if and only if its symplectic quotient manifold is (complete) Kobayashi hyperbolic.
This theorem allows us to construct many new explicit examples of contact-hyperbolic contact manifolds.

Finally, we study the group of contact biholomorphisms. 
Contact hyperbolicity implies that this group is a finite-dimensional Lie group.
For contact \(3\)-manifolds, we sharply bound the dimension of the automorphism group.
We give examples with automorphism groups reaching every possible dimension.
Our third main theorem: the unique maximally symmetric example, up to isomorphism, is the contact manifold $\B^2_{z,w}\times\C_y$ with contact form \(dy+(1-z)^{-2}dw\).
\end{abstract}

\maketitle

\tableofcontents

%
%

\section{Introduction}
The purpose of this paper is to investigate rigidity and invariant metrics in holomorphic contact structures.
Holomorphic contact structures are not as well understood as their real counterparts.
In particular, unlike real contact structures, they are not all equally rigid: the automorphism group of a holomorphic contact structure can have finite or infinite dimension.

Kobayashi defined his pseudometric on a complex manifold from its holomorphic discs.
The holomorphic discs tangent to a contact structure determine the contact structure.
From these discs, Demailly \cite{Dem95, Dem12} constructs a sub-Finsler pseudometric, analogous to the Kobayashi pseudometric.
He integrates this pseudometric along every curve tangent to the contact structure.
This integral is a length for that curve.
The associated pseudodistance between two points is the infimum length of curves tangent to the contact structure and joining the points.
By analogy with Kobayashi hyperbolicity, we call a holomorphic contact structure \emph{contact-hyperbolic} when the pseudodistance is a distance. 
In this paper, we study this pseudodistance in the presence of a global contact form.

When a complex manifold is Kobayashi hyperbolic, every contact structure on it is clearly contact-hyperbolic.
On the other hand, Forstnerič \cite{For17b} described a striking example of a contact-hyperbolic contact structure on complex Euclidean space.
Its automorphism group remains unknown.
On many complex manifolds that are not Kobayashi hyperbolic, we find new examples of contact-hyperbolic contact structures.
We also find the automorphism groups of those contact structures.

The aim of this paper is twofold: 
\begin{enumerate}
\item
We provide a method to construct explicit examples of contact-hyperbolic contact manifolds, particularly with large automorphism groups.
\item
In low dimension, we classify contact-hyperbolic contact manifolds based on the dimension of their automorphism groups. 
\end{enumerate}
This paper has one central insight: we reduce contact hyperbolicity of a large and important family of contact manifolds to classical Kobayashi hyperbolicity via a geometric quotient. 
This observation allows us to replace difficult sub-Finsler arguments on contact manifolds with classical complex-analytic arguments on their quotients. As a consequence, we reduce questions about rigidity and automorphism groups of contact-hyperbolic contact structures to known results in Kobayashi hyperbolicity.

Boothby and Wang \cite{BW} constructed this quotient with its symplectic form in the setting of real contact structures.
Our complex setting gives rise to a new and more sophisticated story, which has no real analogue.
Once we fix the symplectic quotient manifold, we reconstruct all possible contact manifolds with that given quotient, in terms of Dolbeault cohomology.
Only with that reconstruction established does our central insight give us access to automorphism groups, via lifting automorphisms from the quotient.

The automorphism group of a Kobayashi hyperbolic complex manifold is a finite-dimensional real Lie group.
Isaev \cite{Isa07, Isa08} provided a sharp classification of such manifolds 
with biholomorphism group near maximal dimension.
Similarly, the automorphism group of a contact-hyperbolic contact manifold is a finite-dimensional real Lie group.
Via elementary Lie theory, we prove an upper bound on its dimension: see Theorem~\vref{Thm:Riccardo}.
This bound is not sharp.
It does not appear to lead us to a classification in the style of Isaev.
Finding a sharp bound remains an open problem.

\subsection{Statement of the main theorems}
Our work proceeds in three steps: contact geometry, metric geometry and automorphism groups.
Each step yields a main theorem.
To describe our results, we introduce the necessary notation and terminology.

A \emph{contact manifold} \((M,V)\) is a complex manifold \(M\) equipped with a holomorphic contact structure \(V\).
A \emph{Reeb vector field} $v$ on a contact manifold $(M,V)$ is a holomorphic vector field $v$ everywhere transverse to \(V\) but whose flow preserves $V$ \cite[p.~5]{Gei08}.
In particular, a Reeb vector field has no zeroes.
A \emph{Reeb manifold} $(M, V, v)$ is a contact manifold $(M,V)$ equipped with a $\C$-complete Reeb vector field $v$ whose flow acts freely on $M$.
A Reeb manifold is \emph{proper} if the flow of the Reeb vector field acts properly.

Our goal is to reduce problems about a Reeb manifold to problems in lower dimensions, on the quotient space by the flow of the Reeb vector field.
Each Reeb manifold has a canonical contact form: the unique contact form \(\xi\) with \(\ip_v\xi=1\).
The first part of this work proves:
\begin{theorem}[Contact geometry theorem]\label{Thm:main1-intro}
Let $(M,V,v)$ be a proper Reeb manifold.
\begin{itemize}
\item
Bundle structure: the flow of $v$ generates a holomorphic principal \(\C\)-bundle
\[
\begin{tikzcd}
\C\arrow{r}&M\arrow{d}\\
&S,
\end{tikzcd}
\]
with the contact form as a holomorphic connection form.
\item
Symplectic structure: the curvature of this connection is a $C^{\infty}$-exact holomorphic symplectic form \(\omega\) on $S$.
\item
Reconstruction: Conversely, every complex manifold $S$ equipped with a $C^\infty$-exact holomorphic symplectic form \(\omega\) is the quotient of a proper Reeb manifold.
\item
Classification: $H^1(S,\C)$ parameterises the isomorphism classes of proper Reeb manifolds with a given symplectic quotient $(S,\omega)$.
\end{itemize}
\end{theorem}
This theorem summarises Lemma~\vref{lemma:Reeb.to.csl}, Theorem~\vref{thm:C.infty.exact.iff.csl}, Proposition~\vref{prop:flat-conn} and Theorem~\vref{thm:how-many-lifts}.
See Theorem~\vref{thm:how-many-lifts} for a detailed description of how $H^1(S, \C)$ parametrises such classes.
 
Note that two $\C$-principal bundles over a complex manifold $S$ may be isomorphic as $\C$-principal bundles, without being isomorphic as Reeb manifolds over $S$. For instance, if $S$ is Stein, every holomorphic $\C$-principal bundle over $S$ is holomorphically trivial.
Assume in addition that $H^1(S,\C)\neq 0$.
Then the trivial bundle carries nonisomorphic contact structures; see Example~\vref{non-equiv}.

In order to lift a holomorphic symplectic form $\omega$ to a holomorphic contact form on $M$, $\omega$ must be $C^\infty$-exact but need not be holomorphically exact. The following conditions are equivalent:
\begin{itemize}
\item
$\omega$ is holomorphically exact,
\item
the bundle $M\to S$ admits a holomorphic flat connection,
\item
after perhaps replacing $M$ by a covering space, the bundle $M\to S$ is holomorphically trivial.
\end{itemize}
We give additional equivalent conditions in Theorem~\vref{theorem:fit}.

Recall that an \emph{entire curve} is a nonconstant holomorphic map from the complex plane.
Any complex manifold containing an entire curve fails to be Kobayashi hyperbolic.
No Reeb manifold is Kobayashi hyperbolic, because the flow lines of the Reeb vector field are entire curves.
We ask when a Reeb manifold is contact-hyperbolic in Demailly's sense (see section~\vref{section:hyperbolicity}).

A connected contact manifold $(M,V)$ is \emph{(complete) contact-hyperbolic} if the pseudodistance is a (complete) distance.
\begin{theorem}[Metric geometry theorem]\label{Thm:intro-main2}
A proper Reeb manifold is (complete) contact-hyperbolic just when its symplectic quotient is (complete) Kobayashi hyperbolic.
\end{theorem}
We prove this result in Theorems~\vref{Thm:Reeb-principal-bdle} and~\vref{Thm:complete-lift-hyp}.
One direction is easy: if $(M,V)$ is (complete) contact-hyperbolic, then its symplectic quotient $S$ is (complete) Kobayashi hyperbolic, because the $\C$-principal bundle map $M\to S$ does not expand distances. However, the converse is not easy. An analogue of the ball-box theorem appears to be beyond reach, as the metric we associate to $V$ is merely an upper semicontinuous sub-Finsler metric. Instead our proof relies on holomorphic localisation arguments in both \(M\) and \(S\).

Theorem~\ref{Thm:intro-main2} facilitates construction of many nontrivial examples of (complete) contact-hyperbolic contact manifolds, starting from holomorphic symplectic (complete) Kobayashi hyperbolic manifolds.
For instance, if $S$ is any pseudoconvex domain  in $\C^{2N}$, then the standard holomorphic symplectic form on $\C^{2N}$ yields a family of pairwise nonisomorphic contact-hyperbolic contact structures on $S\times \C$, parameterised by $H^1(S,\C)$. 


\begin{theorem}[Automorphism group theorem]\label{Thm:class-3}
Let $(M, V, v)$ be a contact-hyperbolic Reeb manifold with $\dimC M=3$.
Then 
\[
2\leq \dimR \Aut_VM\leq 7.
\]
Each of these dimensions occurs.
Up to isomorphism, the unique Reeb 3-fold with $7$-dimensional automorphism group is
\[
(M,V,v)=(\B^2_{z,w}\times\C_y,[dy+(1-z)^{-2}dw=0],\partial_y),
\]
associated to the symplectic manifold $(S,\omega)=(\B^2,2(1-z)^{-3}dz\wedge dw)$; see Example~\vref{example:big.aut}.
\end{theorem}
This bound contrasts with Kobayashi hyperbolicity, where automorphism groups have dimension up to 15, and manifolds with near maximal symmetry are the most elementary examples: the ball and the bidisc.
The maximal symmetry Reeb manifolds are considerably more intricate.

The proof, in Section~\ref{Sec:class}, also sharpens an earlier bound on $\dimR \Aut_VM$ when $\dimR M=5$. To pin down the base $S$, our proof uses Isaev's classification of Kobayashi hyperbolic manifolds in terms of the dimension of their automorphism groups. This classification allows us to determine, up to biholomorphism, the holomorphic symplectic manifolds that arise from proper Reeb manifolds. 

For example, if $\dimR \Aut_VM=7$ then $\dimR \Aut S\geq 5$. The only possibilities for $S$ are the unit ball or the bidisc, up to biholomorphism. We rule out the bidisc by carefully inspecting the closed Lie subgroups of its automorphism group.

\subsection{Organisation of the paper}
Section~\ref{section:hyperbolicity} defines and develops the notion of contact hyperbolicity.
Section~\ref{section:automorphisms} studies automorphism groups of contact structures.
Sections~\ref{section:contact.symplectic.lifts} to \ref{section:hyperbolicity.of.lifts} show that proper Reeb manifolds arise by lifting holomorphic symplectic structures from quotient manifolds, and develop the moduli spaces.
Sections~\ref{Sec:Lie} to \ref{Sec:class} apply this to reduce classification of automorphism groups to the study of automorphisms of holomorphic symplectic manifolds.

\bigskip

We thank Fabio Podest\`a for providing us with the proof of Theorem~\ref{thm:tilted.subgroups.2}.

\section{Contact hyperbolicity}%
\label{section:hyperbolicity}
In this section, we define
\begin{enumerate}
\item
isotropic discs;
\item
the pseudometric, using isotropic discs;
\item
lengths of isotropic curves, by integrating the pseudometric;
\item
the pseudodistance, as the infimum length of isotropic curves; and
\item
contact hyperbolicity.
\end{enumerate}

\subsection{Contact structures and contact forms}
Let $M$ be a complex manifold of complex dimension $2N+1$, with $N\geq 1$.
\begin{definition}
A \emph{contact form} on $M$ is a holomorphic $1$-form $\xi$ so that $\xi\wedge(d\xi)^N\ne 0$.
A \emph{local contact form} is a contact form on an open subset of $M$.
A \emph{contact structure} $V$ is a holomorphic subbundle of $TM$, of rank $2N$, such that, locally, $V$ is the kernel of a local contact form.
A \emph{contact manifold} $(M,V)$ is a complex manifold $M$ equipped with a contact structure $V$.
\end{definition}
\subsection{The standard example}
On \(\C^{2N+1}\), with coordinates $(z,w,y)\in\C^N\times \C^N\times\C$, the \emph{standard contact form} is
\begin{equation}\label{Eq:standard-contact}
\xi=dy-w_1\,dz_1-\dots-w_N\,dz_N.
\end{equation}
Its kernel is the standard contact structure.
The holomorphic vector field \(\partial_y\) is a Reeb vector field.
A vector field is tangent to the contact structure precisely when it is pointwise in the span of
\[
\partial_{z_1}+w_1\partial_y,\dots,\partial_{z_N}+w_N\partial_y,
\partial_{w_1},\dots,\partial_{w_N}.
\]
Their Lie brackets clearly span $\C^{2N+1}$, since
\[
\left[\partial_{w_j},\partial_{z_j}+w_j\partial_y\right]
=
\partial_y.
\]
Thus the contact structure is one-step bracket generating.
By Sussmann's Orbit Theorem, any two points in any connected open set are connected by a piecewise smooth curve, with each piece the flow line of a smooth vector field tangent to the contact structure \cite{Sus73}.

By the holomorphic Darboux theorem for contact manifolds \cite[p.~66]{Gei08}, every contact structure is locally isomorphic to the standard example, by a local isomorphism carrying a given Reeb vector field to \(\partial_y\). 
Hence, any two points in any connected open set are connected by flow lines of smooth vector fields tangent to the contact planes.
So a contact manifold is connected just when it is connected by piecewise smooth paths tangent to the contact planes.
\subsection{Isotropic discs}
\begin{definition}
Let $(M,V)$ be a holomorphic contact manifold.
A holomorphic map $\varphi\colon C\to M$ from a Riemann surface $C$ is a \emph{$V$-isotropic curve}, or \emph{isotropic curve} or \emph{$V$-curve} for short, if
\[
 d_\zeta\varphi(T_{\zeta} C)\subset V_{\varphi(\zeta)}, \text{ for every $\zeta\in C$.}
 \] 
If $C$ is the unit disc, the $V$-curve is a \emph{$V$-disc}.
\end{definition}
\begin{example}
For the standard contact structure 
\[
\xi=dy-w_1\,dz_1-\dots-w_N\,dz_N,
\]
on $M=\C^{2N+1}$, a $V$-disc consists of holomorphic functions
\[
y=y(\zeta), z_j=z_j(\zeta), w_j=w_j(\zeta), j=1,2,\dots, N,
\]
defined for $|\zeta|<1$, satisfying the holomorphic differential equation
\[
\frac{dy}{d\zeta}=w_1(\zeta)\frac{dz_1}{d\zeta}+\dots+w_N(\zeta)\frac{dz_N}{d\zeta}.
\]
We can construct every $V$-disc in $\C^{2N+1}$ by picking any holomorphic functions 
\[
z_j=z_j(\zeta), w_j=w_j(\zeta), j=1,2,\dots, N,
\]
on the disk, and letting
\[
y(\zeta):=\int 
\left(
w_1(\zeta)\frac{dz_1}{d\zeta}+\dots+w_N(\zeta)\frac{dz_N}{d\zeta}
\right)d\zeta.
\]
Thus, locally, all $V$-curves (not just $V$-discs) arise by explicit quadrature.
\end{example}
\subsection{The pseudometric}
Let $V_0$ be a real or complex vector space.
Let $X\subseteq V_0$ be a subset.
The \emph{Minkowski functional} of \(X\) is
\[
|w|_X:=\inf\set[\frac{1}{\lambda}]{\lambda w\in X,\lambda>0}.
\]
For example, if $X$ is the unit ball of a norm then $|w|_X$ the norm.
Intuitively, $|w|_X$ is large when $w$ is far away from $X$, or more precisely, when a small scale factor is required to scale $w$ into $X$.
As $|w|_X=|w|_{[0,1]X}$, we may assume without loss of generality that \(X\) is closed under rescaling by real numbers in \([0,1]\).
The functional is finite just when \(X\) contains a nonzero vector on every ray from the origin.

Assume that $V_0$ is a complex vector space.
Also assume that $X$ is invariant under rescaling by unit complex numbers (hence by complex numbers in the unit disk).
Then \(w\mapsto|w|_X\) is a \emph{quasi-norm}: $|cw|_X=|c||w|_X$ for all $c\in\C$.

Let $M$ be a complex manifold endowed with a contact structure $V$. 
The \emph{velocity} of a holomorphic disc $\D\xrightarrow\varphi M$ is $\varphi'(0)$.
Let $X_p\subset V_p$ be the set of velocities of \(V\)-discs through \(p\).
Clearly, \(X_p\) is invariant under rescaling by complex numbers in the closed unit disc.
Denote $|v|_{X_p}$ by $\kappa_V(p;v)$, the \emph{Kobayashi--Royden pseudometric} of $V$.
Equivalently,
\[
\kappa_V(p;v)
=\inf\set[\frac{1}{|\lambda|}]{\varphi'(0)=\lambda v, \text{ some $V$-disc } \varphi}.
\]
The pseudometric \(\kappa_V\) measures the scale by which we must divide a vector to get it to become the velocity of an isotropic disc.
The pseudometric \(\kappa_V\) vanishes on a vector just when arbitrarily large dilations of that vector are realised as velocities of isotropic discs.
In particular, if all vectors in \(V\) are velocities of isotropic discs, so velocities are unconstrained, then the pseudometric vanishes.
Clearly, \(\kappa_V(p;cv)=|c|\kappa_V(p;v)\) for \(c\in\C\).
\begin{lemma}\label{lemma:upper.semi}
For any holomorphic contact manifold $(M,V)$, the Kobayashi--Royden pseudometric $\kappa_V$ is a finite, upper semicontinuous, quasi-norm.
\end{lemma}
\begin{proof}
The holomorphic Darboux theorem ensures that, for any $v\in V_p$, there is an isotropic disc with velocity $\lambda v$ for some $\lambda>0$ small enough.
Therefore, $\kappa_V$ is a finite quasi-norm; it is also upper semicontinuous in $p$ \cite[Prop.~1.5, p. 9]{Dem95}.
\end{proof}
%
Let $\D:=\{\zeta\in\C: |\zeta|<1\}$.
Let $k_\D$ denote its hyperbolic distance and $\kappa_\D$ its hyperbolic metric; see, \eg \cite{AbateTaut,Kob}.
Clearly,
\[
\kappa_V(p;v)
=\inf\set[\kappa_\D \left(0;\frac{1}{\lambda}\right)]{\varphi'(0)=\lambda v, \text{ some $V$-disc } \varphi}.
\]
\begin{remark}\label{rem:infint-poinc}
As $\kappa_\D$ is invariant under the group of automorphisms of $\D$, if $\varphi\colon \D\to M$ is a $V$-disc then
\[
\kappa_V(\varphi(\zeta); \varphi'(\zeta))\leq \kappa_\D(\zeta,\partial_\zeta);
\]
an analogue of the Schwarz--Pick lemma.
\end{remark}

\subsection{The pseudodistance}
Assume that $M$ is connected.
Let \(p,q\in M\).
Consider piecewise smooth $V$-tangent curves $\gamma\colon [0,1]\to M$ that join $p$ to $q$.
Sussmann's orbit theorem \cite{Sus73} constructs such curves.
Since \(\kappa_V\) is upper semicontinuous, the function 
\[
t\mapsto\kappa_V(\gamma(t); \gamma'(t))
\]
is upper semicontinuous so bounded and measurable on compact sets \cite{WZ77} p.~56 Corollary~4.16.
The \emph{$V$-length} of $\gamma$ is
\[
\ell_V(\gamma):=\int_0^1 \kappa_V(\gamma(t); \gamma'(t)) dt.
\]
The \emph{$V$-pseudodistance} is
\[
k_V(p,q):=\inf_{\gamma}\ell_V(\gamma).
\]
We take the infimum over \(V\)-tangent curves \(\gamma\) joining \(p\) to \(q\).
The triangle inequality for the pseudodistance follows from concatenating \(V\)-tangent curves.
\begin{lemma}\label{lem:Schwarz}
Let $M$ be a complex manifold endowed with a contact structure $V$. 
Let $\varphi\colon \D\to M$ be a $V$-disc. Then for every $\zeta,\tilde\zeta\in \D$,
\[
k_V(\varphi(\zeta),\varphi(\tilde\zeta))\leq k_\D(\zeta,\tilde\zeta).
\]
\end{lemma}
\begin{proof}
Pick a geodesic $\sigma\colon [0,1]\to \D$ for the Poincar\'e distance such that $\sigma(0)=\zeta$ and $\sigma(1)=\tilde\zeta$. 
Then $\varphi \circ \sigma$ joins \(\varphi(\zeta)\) to \(\varphi(\tilde\zeta)\). 
By definition of \(k_V\) as an infimum,
\begin{align*}
k_V(\varphi(\zeta),\varphi(\tilde\zeta))
\leq 
\ell_V(\varphi \circ \sigma)
=\int_0^1 \kappa_V(\varphi(\sigma(t)); \varphi'(\sigma(t))\sigma'(t))dt.
\end{align*}
By Remark~\ref{rem:infint-poinc},
\[
\int_0^1 
	\kappa_V
	(
		\varphi
		(
			\sigma(t)
		); 
		\varphi'
		(
			\sigma(t)
		)
		\sigma'(t)
	) dt\leq 
\int_0^1 \kappa_\D(\sigma(t); \sigma'(t))dt=k_\D(\zeta,\tilde\zeta).
\]
\end{proof}

\subsection{Continuity}
To establish some continuity of the pseudodistance $k_V$, we need a local lemma:
\begin{lemma}\label{lem:ball-contact}
For $r\in (0,1]$, let $U:=\{p=(z,w,y)\in \C^{2N}\times \C: \|(z,w)\|^2< r^2, |y|<r\}$.
Endow \(U\) with the standard contact structure $V_0$ defined by the standard contact form $\xi$ given in equation \eqref{Eq:standard-contact}. Then 
\[
\lim_{p\to 0}k_{V_0}(p,0)=0.
\]
\end{lemma}
\begin{proof}
We assume that $N=1$; the general case is analogous.

For any $(z,w,y)$ sufficiently close to $(0,0,0)$, $(z,w,y)$ is connected to $(0,0,0)$ by a combination of three $V_0$-curves whose $V_0$-lengths tend to $0$ as $(z,w,y)$ converges to $(0,0,0)$:
\begin{align*}
\gamma_1 (t) =&  (z,t,y), \\
\gamma_2 (t) = & (t,0,y), \\
\gamma_3 (t) = & (\sqrt{y} + t,\sqrt{y}, \sqrt{y}t).
\end{align*}
Travel first along $\gamma_1$ and then $\gamma_2$ to connect $(z,w,y)$ to $(2\sqrt{y},0,y)$, then again $\gamma_1$ to reach $(2\sqrt{y}, \sqrt{y}, y)$. Travel first along $\gamma_3$ to reach $(\sqrt{y}, \sqrt{y}, 0)$, and finally again $\gamma_1$ and $\gamma_2$ to reach the origin.
\end{proof}

\begin{proposition}\label{lem:K_V-continuous}
Let $M$ be a connected complex manifold endowed with a contact structure $V$. 
Then $k_V(p,q)<+\infty$ for every $p,q\in M$ and $k_V\colon M\times M\to [0,+\infty)$ is continuous.
\end{proposition}
\begin{proof}
Lemma~\vref{lemma:upper.semi} ensures finiteness.
By Lemma~\ref{lem:ball-contact}, $\lim_{p\to q}k_V(p,q)=0$, 
as $V$ is locally the standard contact form \eqref{Eq:standard-contact}.
By the triangle inequality, $k_V$ is continuous.
\end{proof}
\subsection{Hyperbolicity of contact structures}
\begin{definition}
A connected contact manifold $(M,V)$ is \emph{contact-hyperbolic} if $k_V$ is a distance, \ie $k_V$ separates points, \ie $k_V(p,q)=0$ if and only if $p=q$.
We say that \((M,V)\) is \emph{complete contact-hyperbolic} if $k_V$ is a complete distance.
\end{definition}
\begin{proposition}\label{prop:same-top}
A connected contact manifold \((M,V)\) is contact-hyperbolic if and only if $k_V$ induces the manifold topology on $M$.
\end{proposition}
\begin{proof}
If $k_V$ induces the manifold topology on $M$ then $k_V$ separates points so $(M,V)$ is contact-hyperbolic. Conversely, suppose $(M,V)$ is contact-hyperbolic.
Assume by contradiction that there is a relatively compact open set $U\subset M$, $p\in U$ and a sequence $p_1,p_2,\dots\in M\setminus\overline{U}$ such that $\lim_{n\to \infty}k_V(p_n,p)=0$. 
For each $n$, pick a $V$-tangent curve $\gamma_n$ from \(p\) to \(p_n\) 
with
\[
\ell_V(\gamma_n)<k_V(p,p_n)+\frac{1}{n}.
\]
By continuity, the curve $\gamma_n$ meets $\partial U$ at a point $q_n$ and
\[
k_V(p,q_n)\leq \ell_V(\gamma_n)<k_V(p,p_n)+\frac{1}{n}.
\]
As $\partial U$ is compact, after passing to a subsequence, $\{q_n\}$ converges, say $q_n \to q_0$, and $q_0\in\partial U$. 
By contact-hyperbolicity, $k_V(p,q_0)>0$.  
But $k_V$ is continuous by Proposition~\ref{lem:K_V-continuous}, hence
\[
0<k_V(p,q_0)=\lim_{n\to\infty} k_V(p,q_n)\leq \lim_{n\to\infty}\ell_V(\gamma_n)<\lim_{n\to\infty}k_V(p,p_n)+\frac{1}{n}=0.
\]
\end{proof}

\subsection{The chain pseudodistance}
We now define a pseudodistance \ala Kobayashi on $M$ using chains of $V$-discs. Take $p, q\in M$.
A \emph{$V$-chain} $\mathcal C=\{(\varphi_j,t_j)\}$ between $p$ and $q$ is a finite set of $V$-discs $\varphi_j\colon \D\to M$, $j=1,\ldots, m$ and real numbers $t_j\in (0,1)$ for every $j\in \{1,\ldots, m\}$, such that $\varphi_1(0)=p$, $\varphi_m(t_m)=q$ and, if $1<j<m$, $\varphi_j(t_j)=\varphi_{j+1}(0)$. 
(We could use complex numbers \(\{t_j\}\), but we can arrange by automorphisms of the disk that they are real.)
The \emph{length} of the chain $\mathcal C$ is
\[
\ell_V(\mathcal C):=\frac{1}{2}\sum_{j=1}^m \log \frac{1+t_j}{1-t_j}=\sum_{j=1}^m k_\D(0,t_j).
\]
By Sussmann's orbit theorem~\cite{Sus73}, some chain of $V$-discs joins any two points of $M$. 
Let \(\tilde k_V(p,q)\) be the infimum length of such chains; $\tilde k_V$ is a pseudodistance on $M$. 
By Lemma~\ref{lem:Schwarz}, for all $p,q\in M$
\begin{equation}\label{Eq:quasi-Barth}
k_V(p,q)\leq \tilde k_V(p,q).
\end{equation}

\begin{remark} Royden \cite{Ro71} shows that the chain and integral Kobayashi pseudodistances coincide; also see Barth~\cite{Bar72}. A key step in the proof is the Lipschitz continuity of the Kobayashi pseudodistance.
Contact pseudodistance is not known to be Lipschitz continuous. Even a ball-box theorem for $k_V$, in the spirit of sub-Riemannian geometry, only implies H\"older regularity of $k_V$. Royden's argument does not apply to $k_V$. It remains an open question whether $k_V=\tilde k_V$. 
\end{remark}

\section{Automorphisms of contact manifolds}%
\label{section:automorphisms}
In this section we study the symmetries of holomorphic contact manifolds. We define contact morphisms and automorphisms.
We introduce the main example that appears throughout the paper.
We establish that the automorphism group of a contact-hyperbolic contact manifold is a finite-dimensional Lie group. 
Finally, we introduce Reeb vector fields and establish their basic properties.
\subsection{Definition of morphisms}
Contact morphisms are the natural maps in contact geometry: they are holomorphic immersions that respect the contact distribution.
\begin{definition}
Take contact manifolds \((M,V)\) and \((M',V')\).
A \emph{contact morphism} \((M,V)\xrightarrow{\Phi}(M',V')\) is a holomorphic immersion \(M\xrightarrow{\Phi}M'\) so that, for every point $p\in M$, $d\Phi_p\left(V_p\right)\subseteq V'_{p'}$ where $p'=\Phi(p)$.
\end{definition}
\begin{lemma}
A contact morphism does not increase the pseudodistance (or the chain pseudodistance) of any two points, and does not increase the Kobayashi--Royden pseudometric on any tangent vector:
\begin{align*}
k_{V'}(\varphi(p),\varphi(q))&\le k_V(p,q),\\
\tilde k_{V'}(\varphi(p),\varphi(q))&\le \tilde k_V(p,q),\\
\kappa_{V'}(\varphi(p),d\varphi_p(v))&\le\kappa_V(p,v),
\end{align*}
for any points $p,q\in M$ and tangent vector $v\in T_p M$.
\end{lemma}
\begin{definition}
Let $M$ be a complex contact manifold with contact structure $V$.
A \emph{$V$-automorphism} of $M$, or \emph{automorphism} of \((M,V)\) or \emph{contact automorphism}, is a bijective contact morphism $(M,V) \to (M,V)$.
It therefore has an inverse contact morphism. We denote by \(\Aut_V M\) the group of automorphisms of the contact manifold $(M,V)$.
\end{definition}
\begin{example}\label{example:big.aut}
We introduce our main example: a contact-hyperbolic contact $3$-fold.
We prove that, up to isomorphism, this is the unique hyperbolic contact $3$-fold with maximal dimensional automorphism group.
Let $M=\B^2_{z,w}\times\C_y$ with contact form
\(\xi=dy+(1-z)^{-2}\,dw\), so \(d\xi=2(1-z)^{-3}dz\wedge dw\).

To describe the automorphism group, it is convenient to pass to the Siegel domain model of \(\B^2\):
\[
\Ha^2:=\set[(Z,W)\in\C^2]{\Re{Z}>|W|^2}.
\]
We also denote this manifold by $S$, because it is the holomorphic symplectic manifold in our theory below.
The Cayley transform \(\B^2\to\Ha^2\),
\[
(z,w)\mapsto F(z,w)=(Z,W)=\left(\frac{1+z}{1-z},\frac{w}{1-z}\right),
\]
is a biholomorphism of the ball with its Siegel domain.
The inverse map is
\[
(z,w)=F^{-1}(Z,W)=\left(\frac{Z-1}{Z+1},\frac{2W}{Z+1}\right).
\]
Under the Cayley transform, \(\xi=dy-\frac{1}{2}W\,dZ+\frac{1}{2}(Z+1)\,dW\), \(d\xi=dZ\wedge dW\).
Consider the map:
\[
(Z,W)\xmapsto{\varphi} (|a|^2 Z + 2 \bar{s}a W + i c +|s|^2, a W + s),
\]
for any \(a \in \C^\ast\), \(c \in \R\) and \(s \in \C\).
We also denote this map by \(\varphi_{a,c,s}\); it is a biholomorphism of $S$, and preserves the holomorphic symplectic form $\omega=dZ\wedge dW$ up to constant factor: 
\[
\varphi_{a,c,s}^\ast\omega=|a|^2a\omega.
\]

Each biholomorphism \(\varphi=\varphi_{a,c,s}\) lifts to a transformation $\Phi=\Phi_{a,c,s,t}$ of $M=S\times\C$:
\[
(Z,W,y)\in M
\xmapsto\Phi
(\varphi(Z,W),Y(Z,W,y)),
\]
for \(t\in\C\) where
\[
Y(Z,W,y)=Y_{a,c,s,t}(Z,W,y)
=
|a|^2ay-\frac{|a|^2s}{2}Z+\frac{a}{2}(c-[|a|^2+|s|^2]+1)W+t.
\]
This map \(\Phi\) is a biholomorphism of $M$ and 
\[
\Phi^*\xi=|a|^2a\xi.
\]
Thus, $\Phi$ preserves the contact structure, and it preserves the contact form precisely when $a=1$.
We prove that the automorphism group of the contact structure $\Aut_V M$ is precisely the set of these maps \(\Phi_{a,c,s,t}\).
Therefore \((M,V)\) has automorphism group of real dimension \(7\).

The Kobayashi pseudometric on the complex unit ball is the one associated to its geometry as complex hyperbolic space, with its Hermitian metric \cite{BL} p.~45, (4.15):
\[
h:=\frac{1}{(1- |z|^2)^2}\bigg(
					(1- |z|^2) \sum_{j=1}^{n} d z_j \otimes d \bar{z}_j
						+ \sum_{j,k=1}^{n} \bar{z}_j z_k d z_j \otimes d\bar{z}_k \bigg)
\]
This pseudometric, applied only the the \(z,w\) coordinates, is also the pseudometric on \(M\), as we will see.
The pseudodistance is not known.
\end{example}
\subsection{The automorphism group as a Lie group}
As in the theory of Kobayashi hyperbolic manifolds, contact hyperbolicity implies that the automorphism group of a contact structure is a finite-dimensional Lie group acting smoothly and properly.
The pseudometric and pseudodistance are invariant under any contact automorphism.
%
\begin{theorem}\label{Thm:Riccardo}
Let $(M,V)$ be a connected contact-hyperbolic complex manifold of complex dimension $2N+1$.
In the compact-open topology, $\Aut_V M$ is a real Lie group of dimension at most $2N^2+5N+4$ acting smoothly and properly on $M$.
\end{theorem}
\begin{proof}
Let $G:=\Aut_V M$ and let $G':=\Aut_{k_V} M$ be the group of biholomorphisms of $M$ preserving the metric $k_V$.
As \(k_V\) is a distance inducing the manifold topology, in the compact-open topology, $G'$ is a Lie group acting smoothly and properly on $M$ \cite[Satz~2.5]{Kau67} .
As \(G\subseteq G'\) is a closed subgroup, it is a Lie subgroup \cite[p.~319 Theorem~9.3.7]{HN2012}, so a Lie group acting smoothly and properly.

We bound the dimension of $G$ by bounding the dimensions of the orbits and of the isotropy subgroups. 
For every $p\in  M$, let $G_p:=\set[\Phi\in\Aut_V M]{\Phi(p)=p}$ be the isotropy of $G$. 
By \cite{Boc45}, $G_p$ is compact and is faithfully represented by its linearised version $L_p:=\{d\Phi_p: \Phi\in G_p\}$. 

As $V_p$ is $L_p$-invariant and $L_p$ is compact, the linear subspace $V_p\subseteq T_p M$ has an $L_p$-invariant complement.
As $V$ is a holomorphic contact structure, each subspace $V_p\subset T_p M$ is a complex hyperplane.
So this complement is of complex dimension $1$, so we denote it by $\C$, writing $T_pM=V_p\oplus\C$. 
A local contact form $\xi$ is uniquely defined up to replacing $\xi$ by $\xi'=e^f\xi$ for any holomorphic function $f$ defined near $p$.
On $V_p$, $\left.d\xi'\right|_p=e^{f(p)}\left.d\xi\right|_p$.
So $V_p$ bears the complex symplectic form $\omega=\left.d\xi\right|_p$, defined up to a nonzero complex scalar. The group $L_p$ acts by isometries preserving $\omega$ up to a complex scale factor when restricted to $V_p$.
Hence $L_p$ is a subgroup of $Sp_N\times \mathbb S^1 \times \mathbb S^1$. 
Here \(Sp_N\) denotes the compact real form of the symplectic group.
As $\dimR Sp_N=N(2N+1)$, $\dimR G_p\leq N (2N+1) +1 + 1=2N^2 + N +2$. As $\dimR M=2(2N+1)$, each orbit \(\mathcal{O}(p)\) has dimension bound $\dimR{\mathcal{O}(p)}\leq\dimR{M}=2(2N+1)=4N+2$. 

Denote by \(\mathfrak{g}_p\subseteq\LieG\) the Lie algebras of \(G_p\subseteq G\).
Each $A\in\mathfrak{g}$ generates an $\R$-complete vector field $X_A$whose flow preserves $V$. 
The Lie algebra \(\mathfrak{g}_p\) consists precisely of the elements  \(A\in \mathfrak{g}\) such that $X_A(p)=0$, while 
\[
\LieG\ni A\mapsto X_A(p)\in T_p \mathcal{O}(p)
\]
is onto, so an exact sequence of linear maps
\[
0\to\LieG_p\to\LieG\to T_p \mathcal{O}(p)\to 0.
\]
Differentiating,
\begin{align*}
\dimR G&=\dimR\LieG, \\ 
\dimR G_p&=\dimR\LieG_p, \\ 
\dimR\mathcal{O}(p)&=\dimR T_p \mathcal{O}(p).
\end{align*}
Hence
\begin{align*}
\dimR G
&=\dimR G_p+\dimR\mathcal{O}(p),
\\
&\leq(2N^2+N+2)+(4N+2),
\\
&=2N^2+5N+4.
\end{align*}
\end{proof}

\subsection{Reeb vector fields}\label{subsection:Reeb}
\begin{lemma}[Geiges {\cite[p.~5]{Gei08}}]\label{Lem:existence-of-local-Reeb}
Let $(M,V)$ be a holomorphic contact manifold. 
For each Reeb vector field \(v\) of  \(V\) there is a unique global contact form \(\xi\) so that $\ip_v{\xi}\equiv 1$.
Conversely, if \(\xi\) is a global contact form for $V$, there is a unique Reeb vector field $v$ for $V$ such that \(\xi(v)\equiv 1\).
\end{lemma}
By the Cartan formula
\[
\LieDer_v\xi=\ip_v{d\xi}+d(\ip_v{\xi}),
\]
every Reeb vector field is a symmetry of its contact form.
Conversely, it is the unique symmetry of its contact form with $\ip_v{\xi}=1$.
\begin{remark}
Suppose that $V$ is a holomorphic contact structure on a complex manifold $M$ and that $v$ is a $\C$-complete Reeb vector field for $V$.
So $v(p)\neq 0$ for all $p\in M$.
Therefore the flow of $v$ acts locally freely, but perhaps not freely.
At each point $p\in M$, the set of times $t\in\C$ at which the flow of $v$ takes $p$ to itself form a discrete subgroup of $\C$.
Every vector field is invariant under its own flow, so that subgroup is the same for all $p$ in any flow line of $v$.
So $(M,V,v)$ is a Reeb manifold just when that subgroup is $\set{0}$ for all $p\in M$, that is, when no point returns to itself after flowing for a nonzero complex time.
The methods of this paper extend readily to allow for a nontrivial subgroup, as long as it is the same at every point $p\in M$.
\end{remark}
In a contact-hyperbolic contact structure, Reeb vector fields are proper:
\begin{proposition}\label{prop:3bound}
Let $(M,V,v)$ be a contact-hyperbolic Reeb manifold of complex dimension $2N+1$.
Then it is proper and
\[
2\leq \dimR \Aut_V M\leq 2N^2+5N+4.
\]
\end{proposition}
\begin{proof}
The flow acts properly by Theorem~\ref{Thm:Riccardo}.
The upper bound is given by Theorem~\ref{Thm:Riccardo}. 
The lower bound arises from the flow of the Reeb vector field.
\end{proof}
\begin{remark}
Forstnerič \cite{For17b} proves the existence of a contact-hyperbolic contact structure $V$ on $\C^3$. His structure does not admit a \(\C\)-complete Reeb vector field: by Theorem~\ref{Thm:Reeb-principal-bdle}, $\C^3$ would have a holomorphic fibre bundle map to a Kobayashi hyperbolic complex surface $S$.
But every holomorphic map from a complex Euclidean space of any dimension to a Kobayashi hyperbolic manifold is constant.
We conjecture that the group of automorphisms of this contact manifold is discrete.
\end{remark}

\section{Contact symplectic lifts}%
\label{section:contact.symplectic.lifts}
In this section, we show that every proper Reeb manifold, after we quotient by the Reeb flow, becomes the total space of a holomorphic principal bundle over a holomorphic symplectic manifold.
The resulting pair of a contact manifold with its symplectic quotient is a \emph{contact symplectic lift}.
We prove that the notion of contact symplectic lift is the essential bridge between contact manifolds and symplectic manifolds, from which we can relate their pseudometrics.
\subsection{Definition of contact symplectic lifts}
\begin{definition}\label{Def:contact-sympl-lift}
A \emph{contact symplectic lift} \((M,V,v,S,\omega)\) is: 
\begin{enumerate}
\item 
a Reeb manifold $(M,V,v)$, say with associated contact form \(\xi\) and 
\item
a holomorphic symplectic manifold $(S,\omega)$ and 
\item 
a $\C$-principal bundle $\pi\colon M\to S$, whose \(\C$-action is the flow of the Reeb vector field, so that
\item $d\xi=\pi^\ast\omega$.
\end{enumerate}
We denote the contact symplectic lift by the data $(M, V, v, S, \omega)$, with the projection map $\pi$ understood.
\end{definition}
\[
\begin{tikzpicture}[
  double/.style={
  		draw, 
  		anchor=north,
  		rectangle split,
  		rectangle split parts=2,
  		rounded corners,
  		rectangle split draw splits=false},
	BlankBox/.style={
		double,
		fill=white,
		draw=white,
		text=white},
	ReebBox/.style={
		double,
		fill=gray!15,
		draw=gray},
	ProperBox/.style={
		double,
		fill=gray!35,
		draw=gray,
		text=black},
	HyperbolicBox/.style={
		double,
		fill=gray!55,
		draw=gray,
		text=black},
	CompleteBox/.style={
  		anchor=text, 
		rectangle,
  		rounded corners,
		fill=gray!75,
		draw=gray,
		text=black},
  ]
  \coordinate (Reeb) at (0,0);
  \node[ReebBox,align=center] at (Reeb) {%
  Reeb manifolds
    \nodepart{second}
      \tikz{\node[ProperBox,align=center] (ProperReebNode) at (Reeb) {%
      proper
      \nodepart{second}
      	\tikz{\node[HyperbolicBox,align=center] {%
     	 		hyperbolic
   	   \nodepart{second}
      		\tikz{\node[CompleteBox,align=center] {%
     	 			complete   	   
		      };
	      	}      
   	   
	      };
      	}      
      };
      }
  };
  \coordinate (CSL) at ($ (ProperReebNode.north east)+(3cm,0) $);
	\node[BlankBox,align=center] (CSLNode) at (CSL) {
	Reeb manifolds
	\nodepart{second}
	\tikz{
  \node[ProperBox,align=center] (InnerCSLNode) at (CSL) {%
	contact symplectic lifts
    \nodepart{second}
      	\tikz{\node[HyperbolicBox,align=center] {%
     	 		hyperbolic
   	   \nodepart{second}
      		\tikz{\node[CompleteBox,align=center] {%
     	 			complete   	   
		      };
	      	}      
   	   
      };
      }
  };
  }
  };
  \path[very thick,stealth-stealth] 
  (ProperReebNode.east|-CSLNode.west)
  edge 
  (CSLNode.west);
  \coordinate (symplectic) at ($ (CSLNode.south)+(0,-.5) $);  \node[ProperBox,align=center] (symplecticNode) at (symplectic) {%
	$C^{\infty}$-exact holomorphic symplectic manifolds
    \nodepart{second}
      	\tikz{\node[HyperbolicBox,align=center] {%
     	 		hyperbolic
   	   \nodepart{second}
      		\tikz{\node[CompleteBox,align=center] {%
     	 			complete   	   
		      };
	      	}      
   	   
      };
      }
  };
  \path[very thick,-stealth] 
  (CSLNode.south)
  edge 
  (symplecticNode.north);
\end{tikzpicture}
\]
\begin{remark}
Because \(d\xi=\pi^\ast\omega$, the total space $M$ has complex dimension $2N+1$, while the base $S$ has complex dimension $2N$, for a unique integer $N\ge 1$.
\end{remark}
\begin{remark}
The fibres of  $\pi\colon M\to S$ are contractible, so $M$ is homotopically equivalent to $S$. In particular, $M$ is contractible if and only if $S$ is.
\end{remark}
Given a \(C^\infty\) (holomorphic) map \(M\to S\) of manifolds, a \(C^{\infty}\) (holomorphic) differential form \(\theta_M\) on \(M\) is \emph{basic} if it is the pullback of a \(C^{\infty}\) (holomorphic) differential form \(\theta_S\) on \(S\).
If \(M\to S\) is a \(C^{\infty}\) (holomorphic) surjective submersion, pullback is injective.
So \(\theta_S\) is then unique.
A \emph{vertical vector field} is a $C^\infty$ (holomorphic) vector field $v$ defined on an open subset of $M$, tangent to the fibres of $M\to S$.
A differential form \(\theta_M\) on \(M\) is \emph{semibasic} if, for any  vertical vector field \(v\), $0=\ip_v\theta_M$.
Every basic form is semibasic.
\begin{lemma}[Libermann \& Marle p.~57 Proposition~3.6, Ivey \& Landsberg p.~339]\label{lemma:semibasic}
Suppose that $M\to S$ is a $C^\infty$ (holomorphic) submersion with connected fibres.
A (holomorphic) differential form on \(M\) is basic if and only if both the form and its exterior derivative are semibasic.
\end{lemma}
\begin{proof}
This is well known, but we prove it for completeness.

Suppose that $v$ is a $C^\infty$ (holomorphic) vector field tangent to the fibres of $M\to S$.
By the Cartan formula, 
\[
\LieDer_v\theta_M=\ip_v{d\theta_M}+d(\ip_v{\theta_M})=0,
\]
the two vanishing conditions force $\theta_M$ to be invariant under the flow of $v$.

Suppose we can solve the problem locally.
Take an open cover \(\{S_\alpha\}\) of \(S\) and an open cover \(\{M_\alpha\}\) of \(M\).
Suppose that \(M\to S\) restricts to each \(M_\alpha\) to a surjective submersion \(M_\alpha\to S_\alpha\) with connected fibres.
Let \(\theta_{M_\alpha}\) be the restriction of \(\theta_M\) to \(M_\alpha\).
Since we assume we can locally solve the problem, we have a form \(\theta_{S_\alpha}\) on \(S_\alpha\) pulling back to \(\theta_{M_\alpha}\).

Consider the flow \(\Phi_t\) of some vertical vector field, say defined on an open set \(U\).
That flow preserves \(\theta_M\).
Since the flow commute with projection to \(S\), the same \(\theta_{S_\alpha}\) form pulls back to \(\Phi_t(M_\alpha\cap U)\) to agree with \(\theta_M\).
So we can expand the open set \(M_\alpha\) to be invariant under all flows of vertical vector fields.
Since the fibers of \(M\to S\) are connected, we can assume that \(M_\alpha=\pi^{-1}S_\alpha\).

Let \(\theta_{S_{\alpha\beta}},\theta_{M_{\alpha\beta}}\) be the restrictions of \(\theta_{S_\alpha},\theta_{M_\alpha}\) to \(S_{\alpha\beta}:=S_\alpha\cap S_\beta\) and \(M_{\alpha\beta}:=M_\alpha\cap M_\beta=\pi^{-1}S_{\alpha\beta}\).
So both \(\theta_{S_{\alpha\beta}}\) and \(\theta_{S_{\beta\alpha}}\) pullback to \(\theta_{M_{\alpha\beta}}=\theta_{M_{\beta\alpha}}\).
By injectivity of pullback, \(\theta_{S_{\alpha\beta}}=\theta_{S_{\beta\alpha}}\).
The differential forms on open sets of a manifold are a sheaf.
Define \(\theta_S\) by 
\[
\theta_S:=\theta_{S_\alpha} \quad \text{ on \(S_\alpha\)}.
\]

The claim is therefore local on $S$ and $M$.
We can pick (holomorphic) coordinates in which $\pi$ has the form
\[
(x,y)\xmapsto{\pi}x.
\]
If $I=(i_1,\dots,i_p)$, let
\[
dx^I=dx^{i_1}\wedge\dots\wedge dx^{i_p},
\]
and so on.
Write
\[
\theta_M=\sum a_{IJ}(x,y)dx^I\wedge dy^J.
\]
Take $v=\partial_{y^j}$ and expand out $0=\ip_v{\theta_M}$ to check that there are no $dy^j$ factors:
\[
\theta_M=\sum a_I(x,y)dx^I.
\]
Expand out $0=\ip_v{\theta_M}$ to check that $a_I(x,y)=a_I(x)$:
\[
0=\ip_v{d\theta_M}=\sum\frac{\partial a_I}{\partial y^j}dx^I.
\]
The form $\theta_M$ is invariant under the flow of $v$, where defined, and is the pullback of a form from $S$.
\end{proof}

\begin{lemma}\label{lemma:Reeb.to.csl}
A Reeb manifold \((M,V,v)\) is proper if and only if it admits a contact symplectic lift.
\end{lemma}
Foreman\cite[p.~187 Theorem 6.2]{For00} proves a similar result, with the same proof, but assuming compact fibres and base.
For real contact structures, Boothby and Wang \cite{BW} have a similar result and proof. 
\begin{proof}
If \((M,V,v,S,\omega)\) is a contact symplectic lift, then \((M,V,v)\) is clearly a proper Reeb manifold.
Suppose that \((M,V,v)\) is a proper Reeb manifold.
The flow of \(v\) is a free and proper holomorphic \(\C\)-action.
Let $S:=M/\C$ and denote the projection by $\pi\colon  M\to S$: $\pi(p)=\C p$.
The map $\pi\colon M\to S$ is a holomorphic $\C$-principal bundle \cite[p.~286, Proposition~4.1.23]{AM1978} for a unique holomorphic structure on $S$.

The contact form \(\xi\) yields a holomorphic \(2\)-form \(d\xi\) invariant under the flow of the Reeb vector field $v$.
By the Cartan formula,
\[
0=\LieDer_v\xi=\ip_v{d\xi}+d(\ip_v\xi)=\ip_v{d\xi}.
\]
So \(\ip_v d\xi=0\).
By lemma~\vref{lemma:semibasic}, \(d\xi\) is the pullback of a unique holomorphic \(2\)-form \(\omega\) on\(S\).
Pullback: $\pi^*d\omega=d\pi^*\omega=d(d\xi)=0$, so $d\omega=0$.

%
Let $N:=(1/2)\dimC{S}$; $\omega$ is non-degenerate because
\[
\xi\wedge\pi^*(\omega^N)=\xi\wedge (d\xi)^N\neq 0.
\]
Thus, $(M, V, v, S, \omega)$ is a contact symplectic lift.
\end{proof}

\begin{lemma}\label{lemma:csl.iso}
Suppose that \((M,V,v)\) is the Reeb manifold of two contact symplectic lifts \((M,V,v,S,\omega)\), \((M,V,v,S',\omega')\).
Then there is a unique biholomorphism \(S\xrightarrow{\varphi}S'\) so that the diagram
\[
\begin{tikzcd}
&M\arrow[dl]\arrow[dr]&\\
S\arrow[rr,"\varphi"]&&S'
\end{tikzcd}
\]
commutes and so that $\varphi^*\omega'=\omega$.
\end{lemma}
\begin{proof}
The two lifts have the same vector field $v$.
The fibres of $M\to S$ and $M\to S'$ are the orbits of $v$.
So the maps $M\to S$ and $M\to S'$ have the same fibres.
The map $M\to S'$ is constant on the fibres of $M\to S$.
So the map $M\to S'$ descends to a map $\varphi\colon S\to S'$, and vice versa, so $\varphi$ is a bijection.

The map \(M\to S\) is a holomorphic submersion, so admits local holomorphic sections near any point of \(S\).
Compose any section with \(M\to S'\) to see that \(\varphi\colon S\to S'\) is holomorphic.
By the same reasoning, \(\varphi^{-1}\colon S'\to S\) is holomorphic.
So \(\varphi\) is a biholomorphism.
The form $\omega$ is the unique holomorphic symplectic form with \(\pi^*\omega=d\xi$, because pullback of forms by any submersion is injective.
So $\varphi^*\omega'=\omega$.
\end{proof}
\begin{definition}\label{def:iso.csl}
An \emph{isomorphism} of contact symplectic lifts
\[
(M,V,v,S,\omega)\to (M',V',v',S',\omega')
\]
is an isomorphism of their associated Reeb manifolds 
\[
(M,V,v)\to (M',V',v'),
\]
to which we attach the induced symplectic biholomorphism 
\[
(S,\omega)\to (S',\omega').
\]
\end{definition}
\begin{example}
We return to example~\vref{example:big.aut}: $M=\B^2_{z,w}\times\C_y$ with contact form 
\[
\xi=dy+(1-z)^{-2}dw.
\]
Let $S:=\B^2_{z,w}$ and note that
\[
\omega=d\xi=\frac{2}{(1-z)^3}dz\wedge dw
\]
is a holomorphic symplectic form on $S$.
\end{example}
\subsection{Local contact symplectic coordinates}
We prove a Darboux theorem for contact symplectic lifts: a local normal form.
\begin{definition}\label{definition:trivial}
A contact symplectic lift $(M, V, v, S, \omega)$ is \emph{trivialised} if
\begin{itemize}
\item
$M=S\times\C$ as a holomorphically trivial principal $\C$-bundle and
\item
the bundle map is $\pi(x,y)=x$ for $(x,y)\in S\times\C$ and 
\item
the bundle action is translation in $y$, so $v=\partial_y$.
\end{itemize}
A contact symplectic lift $(M, V, v, S, \omega)$ is \emph{trivial} if it admits a \emph{trivialisation}, \ie an isomorphism with a trivialised contact symplectic lift.
In other words, by isomorphisms of the contact manifolds, and of the symplectic manifolds, we match the principal $\C$-actions and the Reeb vector fields.
\end{definition}
\begin{definition}
Take a holomorphic symplectic manifold $(S,\omega)$.
A \emph{holomorphic symplectic potential} is a holomorphic $1$-form $\nu$ on $S$ with $d\nu=\omega$.
\end{definition}
\begin{definition}
Let $(M,V,v,S,\omega)$ be a contact symplectic lift with holomorphic symplectic potential $\nu$.
The contact symplectic lift is \emph{$\nu$-fit} if the bundle $M\to S$ admits a trivialisation in which, in the notation of Definition~\ref{definition:trivial}, the contact form $\xi$ on $M$ associated to the Reeb vector field $v=\partial_y$ is $\xi=dy+\pi^*\nu$. 
A contact symplectic lift $(M,V,v,S,\omega)$ is \emph{fit} if it is $\nu$-fit for some holomorphic symplectic potential $\nu$.
\end{definition}
Different holomorphic symplectic potentials may yield different trivialisations. However we have the following preliminary result:
\begin{lemma}[Independence of potential]\label{lemma:nu.fit}
Take a contact symplectic lift $(M,V,v,S,\omega)$ and two holomorphic symplectic potentials $\nu$ and $\nu'$.
Suppose that $H^1(S,\C)=0$.
Then the contact symplectic lift is $\nu$-fit if and only if it is $\nu'$-fit.
\end{lemma}
\begin{proof}
As $d\nu=d\nu'=\omega$, $d(\nu'-\nu)=0$.
But $H^1(S,\C)=0$ so $\nu'=\nu+dh$ for some $C^\infty$ function $h\colon S\to\C$.
As $dh=\nu'-\nu$ is $(1,0)$, $h$ is holomorphic.
Write $\xi=dy+\pi^*\nu$ and set $y'(x,y):=y-h(x)$; check that 
\[
dy'+\pi^*\nu'=dy-d\pi^*h+d\pi^*h+\pi^*\nu=dy+\pi^*\nu.
\]
So \((x,y)\mapsto (x,y+h(x))\) trivialises $M$ as a $\nu'$-fit contact symplectic lift.
\end{proof}
\begin{definition}
Take a holomorphic symplectic manifold $(S,\omega)$ of complex dimension $2N$. \emph{Holomorphic Darboux coordinates} on an open subset $U\subseteq S$ are local holomorphic coordinates $(z,w)=(z_1,\ldots, z_N, w_1,\ldots, w_N)$ on $U$ such that
\begin{equation}\label{Eq:Darboux-down-for-lift}
\omega|_U=d\nu, \quad \nu:=-\sum_{j=1}^N w_j dz_j,
\end{equation}
\ie $\nu$ is a holomorphic symplectic potential on $U$.
\end{definition}
The holomorphic Darboux theorem \cite[section 8.1]{Cannas}, \cite{Darb82a},\cite{Darb82b}, \cite[section 22]{GS1990}, \cite[section 3.2]{McDuff17}, \cite{Moser} ensures that every point of any holomorphic symplectic manifold, lies in the domain of holomorphic Darboux coordinates.
We now extend the Darboux theorem to contact symplectic lifts; a global version appears in Theorem~\vref{theorem:fit}.
\begin{lemma}[Local contact symplectic coordinates]\label{Lem:contact-sympl-coordinates}
Take a contact symplectic lift $(M, V, v, S, \omega)$.
Every point of $S$ lies in a contractible Stein open set $U\subseteq S$.
We can pick \(U\) to lie in the domain of Darboux coordinates $(z,w)$.
Hence \(\omega=d\nu\) where
\[
\nu=-\sum_{j=1}^N w_j\,dz_j.
\]
Pull back these coordinates to the preimage $\pi^{-1}U\subseteq M$ of that domain.
There is a holomorphic function $y$ on \(\pi^{-1}U\) so that $(z,w,y)$ are coordinates in which the contact form of the Reeb manifold \((M,V,v)\) is \(\xi=dy+\nu\).
The Reeb vector field is \(v=\partial_y\).
Hence \(\pi^{-1}U\to U\) is a trivialised contact symplectic lift.
\end{lemma}
\begin{proof}
Choose a point of $S$, Darboux coordinates \((z,w)\) near that point, and a contractible Stein open set $U$ in the domain of those coordinates. 
We can replace $S$ by that open set, $M$ by its preimage.
So assume that $S$ is contractible and Stein.
Every holomorphic principal bundle on $S$ is equivariantly holomorphically trivial \cite[p.~356 Theorem 8.2.1]{For17}.
Up to a $\C$-equivariant biholomorphism, $M=S\times\C_y$ and
$z,w,y$ are holomorphic coordinates with
\[
\omega=d\nu, \quad \nu:=-\sum_{j=1}^N w_j dz_j.
\]
and $\C$ acts by translation in $y$.
The projection $\pi$ is $\pi(z, w, y)=(z, w)$.
The flow of the Reeb vector field is the $\C$-action, so $v=\partial_y$, and \(d(\xi-(dy+\nu))=0\).
Since \(M\) is contractible, \(H^1(M,\C)=0\) so there is a holomorphic function $h$ on $M$ so that 
\[
\xi=dy+\nu+dh=d(y+h)+\nu.
\]
But \(\LieDer_v h=\ip_v{dh}=\ip_v(\xi-dy-\nu))=0\), so $h$ is constant on the fibres of $M\to S$, i.e. \(h=h(z,w)\).
Let \(\tilde{y}:=y+h\).
So \((z,w,\tilde{y})\) are coordinates, and $v=\partial_{\tilde{y}}$.
\end{proof}
\begin{definition}
Such coordinates \((z,w,y)\) are \emph{contact symplectic coordinates}.
\end{definition}
\begin{corollary}\label{cor:autos.id.on.S}
Take a contact symplectic lift $(M, V, v, S, \omega)$ with \(S\) connected.
Every automorphism of the Reeb manifold $(M,V,v)$ acts on $S$ as a biholomorphism.
It descends to the identity on $S$ if and only if it is the time $t$ flow of the Reeb vector field for some $t\in\C$.
\end{corollary}
\begin{proof}
The flow of $v$ is the $\C$-bundle action.
A $V$-automorphism leaves $v$ invariant precisely when acting as a $\C$-bundle automorphism, hence acting on $S$ by biholomorphism.
Acting on $S$ trivially, the $V$-automorphism acts trivially on each open set $U\subseteq S$, hence acts preserving its preimage $\pi^{-1}U\subseteq M$.
We then see directly that the $V$-automorphism has the form 
\[
(z,w,y)\mapsto (z,w,y+h(z,w)).
\]
Preserving $v=\partial_y$ forces constant $h$.
\end{proof}
\subsection{Pullback of principal bundles}
We recall the standard construction of holomorphic principal bundles.
\begin{definition}
Let
\[
\begin{tikzcd}
G\arrow[r]&M'\arrow[d,"\pi'"]\\
&S'
\end{tikzcd}
\]
be a holomorphic principal bundle.
Take a holomorphic map $S\xrightarrow{\varphi}S'$.
The \emph{pullback bundle} has total space 
\[
M:=\set[(s,p')\in S\times M']{\varphi(s)=\pi'(p')},
\]
map $(s,p')\in M\xmapsto{\pi}s\in S$, and group action \(g(s,p')=(s,gp')\).
\begin{lemma}
The pullback bundle is a holomorphic principal bundle.
\end{lemma}
\begin{proof}
The proof is the same as for \(C^{\infty}\)-bundles \cite[Appendix~A, p.~374]{LM}.
Clearly, $M\subseteq S\times M'$ is a Zariski-closed subset.
Transversality of $\varphi$ to $\pi'$ makes $M\subseteq S\times M'$ a closed embedded complex submanifold.
The group action is holomorphic, free and proper on $S\times M'$.
The complex submanifold $M$ is invariant, so this $G$-action restricts to a holomorphic, free and proper $G$-action on $M$.
The map $\pi$ is invariant under this $G$-action, so descends to a biholomorphic map $M/G\to S$, so $S\cong M/G$ is the quotient space of this action and
\[
\begin{tikzcd}
G\arrow[r]&M\arrow[d,"\pi"]\\
&S
\end{tikzcd}
\]
is a holomorphic principal $G$-bundle.
\end{proof}
Define a holomorphic map $M\xrightarrow{\Phi}M'$, the \emph{pullback bundle morphism} by 
\[
(s,p')\in M\mapsto \Phi(s,p'):=p'
\]
giving a commutative diagram
\[
\begin{tikzcd}
M\arrow[d,"\pi"']\arrow[r,"\Phi"]&M'\arrow[d,"\pi'"]\\
S\arrow[r,"\varphi"]&S',
\end{tikzcd}
\]
so \(\Phi\) is a morphism of holomorphic principal $G$-bundles.
If, in addition, $S\xrightarrow{\varphi}S'$ is an immersion (respectively, a submersion) then $M\xrightarrow{\Phi}M'$ is also an immersion (respectively, a submersion).
\end{definition}
\subsection{Lifting holomorphic discs}
Holomorphic curves in the symplectic base lift to isotropic curves upstairs:
\begin{proposition}\label{Prop:V-lifts}
Let $(M, V, v, S, \omega)$ be a contact symplectic lift with projection $\pi\colon M\to S$. 
Take a holomorphic map $\varphi\colon C \to S$  from a connected Riemann surface.
Take points $\zeta_0\in C$, $p_0\in M$ so that $\pi(p_0)=\varphi(\zeta_0)$.
Take the universal covering space $(\tilde{C},\tilde{\zeta}_0)\to (C,\zeta_0)$.
There is a unique $V$-curve $\tilde\varphi\colon \tilde{C}\to M$ such that $\tilde\varphi(\tilde{\zeta}_0)=p$ and $\pi \circ \tilde\varphi=\varphi$.
\end{proposition}
\begin{proof}
%

Choose a Stein contractible neighborhood of $\bar p_0:=\pi(p_0)$ in $S$.
Pick contact symplectic coordinates as in Lemma~\ref{Lem:contact-sympl-coordinates}.
The map $\varphi$ locally consists of holomorphic functions $z_j=z_j(\zeta)$, $w_j=w_j(\zeta)$, $j=1,\ldots, N$.
Construct the pullback bundle
\[
\varphi^*M:=\set[(\zeta,p)]{\zeta\in C, p\in M, \varphi(\zeta)=\pi(p)}.
\]
The pullback bundle is a holomorphic principal $\C$-bundle.
The global contact form $\xi$ defined by $\xi(v)\equiv 1$ and $\xi(V)=0$ satisfies $\mathcal{L}_v\xi=0$ and $\ip_v\xi=1$.
It is a holomorphic  connection on $\varphi^*M$.
Its curvature is a holomorphic $2$-form on a Riemann surface $C$, so vanishes: $\xi$ is a flat holomorphic connection.
On some covering space $(\tilde{C},\tilde{\zeta}_0)\xrightarrow{\pi}(C,\zeta_0)\), the $\C$-bundle is holomorphically trivial, with $\xi$ identified with the standard flat holomorphic connection $dy$ \cite[Theorem~8,  p.~201]{Atiyah57}:
\[
\begin{tikzcd}
\tilde{C}\times\C\arrow[r]\arrow[d]&\varphi^*M\arrow[d]\\
\tilde{C}\arrow[r]&C.
\end{tikzcd}
\]
We can assume this covering space is the universal covering space.
The level sets of $y$ are curves in $\tilde{C}\times\C$, lifting $\tilde{C}$ into $\tilde{C}\times\C$, satisfying $\xi=0$, and foliating $\tilde{C}\times\C$.
The equation $\tilde\varphi(\tilde\zeta_0)=p_0$ singles out one leaf of that foliation.
\end{proof}
For example, there are entire curves in the base symplectic manifold just when there are entire $V$-curves upstairs.
\section{Exactness}
In this section, we study exactness properties of the holomorphic symplectic form in the base of a contact symplectic lift. Every contact symplectic lift has a $C^\infty$-exact symplectic form.
Holomorphic exactness is equivalent to vanishing of a cohomological obstruction.
\subsection{%
\texorpdfstring%
{\u{C}ech--de~Rham description of exactness}%
{Cech-de Rham description of exactness}%
}\label{section:De-Rham-coh}
We explain some elementary \u{C}ech--de~Rham theory.
We use it repeatedly below.
Take a $C^{\infty}$-exact holomorphic $2$-form $\mu$ on a complex manifold $S$.
We define a cohomology class $\cls\mu$ so that $\mu$ is holomorphically exact precisely when $\cls\mu=0$.

Choose a good open cover $\{S_\alpha\}$.
So every finite intersection $S_{\alpha\beta\dots}=S_\alpha\cap S_\beta\cap \dots$ is contractible \cite[Theorem~5.1]{Bo-Tu}.
We can also assume that every finite intersection is Stein.
By the de~Rham and Leray theorems \cite[p.~209 Theorem~5.17]{Dem12},
\[
H^2_{dR}(S)\cong H^2(S,\C)\cong \check H^2(\{S_\alpha\}, \C).
\]
By the holomorphic Poincar\'e lemma \cite[p.~448 and p.~451]{GH78}, there is a
holomorphic $1$-form $\nu_\alpha$ on each $S_\alpha$ with $d\nu_\alpha=\mu$. 
On each overlap $S_{\alpha\beta}$, $\nu_\beta-\nu_\alpha$ is closed.
Thus it is holomorphically exact, again by the holomorphic Poincar\'e lemma.
So there are holomorphic functions \(f_{\alpha\beta}\) on each \(S_{\alpha\beta}\) so that
\[
df_{\alpha\beta}=\nu_\beta-\nu_\alpha.
\]
On each triple overlap \(S_{\alpha\beta\gamma}\), the sum
\[
c_{\alpha\beta\gamma}:=f_{\alpha\beta}+f_{\beta\gamma}+f_{\gamma\alpha}
\]
is constant.
The class $[\{c_{\alpha\beta\gamma}\}]\in H^2(S, \C)$ represents $[\mu]$ in $H^2(S, \C)$ \cite[Proposition~8.8]{Bo-Tu}.

Consequently, $\mu$ is $C^{\infty}$-exact just when this class vanishes, \ie just when there are constants $c_{\alpha\beta}$ so that
\[
c_{\alpha\beta\gamma}:=c_{\alpha\beta}+c_{\beta\gamma}+c_{\gamma\alpha}.
\]
Replacing $f_{\alpha\beta}$ by $f_{\alpha\beta}-c_{\alpha\beta}$,
\[
0=f_{\alpha\beta}+f_{\beta\gamma}+f_{\gamma\alpha}.
\]
So $\set{f_{\alpha\beta}}$ defines a holomorphic $1$-cocycle, defining a class
\[
[\set{f_{\alpha\beta}}]\in H^1(S,\mathcal{O}).
\]
Denote the quotient of this class by \(\iota_*H^1(S,\C)\) as
\[
\cls\mu\in H^1(S,\mathcal{O})/\iota_* H^1(S,\C).
\]

Conversely, suppose that we alter the choice of \(f_{\alpha\beta}\), say to \(f'_{\alpha\beta}\), and \(\nu_\alpha\) to some \(\nu'_\alpha\).
So 
\[
df'_{\alpha\beta}=\nu'_{\beta}-\nu'_{\alpha},
\]
and \(d\nu'_\alpha=\mu\).
So \(\nu'_\alpha-\nu_\alpha\) is closed, so exact, say 
\[
\nu'_\alpha-\nu_\alpha=df_\alpha,
\]
for some holomorphic functions \(f_\alpha\).
The functions
\[
c_{\alpha\beta}:=f'_{\alpha\beta}-f_{\alpha\beta}-f_\beta+f_\alpha.
\]
are constants and satisfy the cocycle condition.
Hence the difference class
\[
[\{c_{\alpha\beta}\}]=[\{f'_{\alpha\beta}-f_{\alpha\beta}\}]
\]
is a class \(\) of constants.
Denote the inclusion of the complex numbers into the sheaf of holomorphic functions on a complex manifold by \(\C\xrightarrow\iota\mathcal{O}\):
\[
\iota_*[\{c_{\alpha\beta}\}]
\in
\iota_*H^1(S,\mathcal{O})
\subseteq
H^1(S,\mathcal{O}).
\] 
So the difference class is in the image of 
\[
H^1(S,\C)\xrightarrow{\iota_*}H^1(S,\mathcal{O}).
\]
In particular, \(\cls\mu\in H^1(S,\mathcal{O})/\iota_*H^1(S,\C)\) is well defined.

Conversely, fix potentials \(\nu_\alpha\) and functions \(f_{\alpha\beta}\).
Pick any constants \(c_{\alpha\beta}\) satisfying the cocycle condition 
\[
0=c_{\alpha\beta}+c_{\beta\gamma}+c_{\gamma\alpha},
\]
and any holomorphic function \(f_\alpha\) on each \(S_\alpha\).
Let
\[
\nu'_\alpha:=\nu_\alpha+df_\alpha,
\]
and
\[
f'_{\alpha\beta}:=c_{\alpha\beta}+f_{\alpha\beta}+f_\beta-f_\alpha.
\]
The class \([\{f_{\alpha\beta}\}]\in H^1(S,\mathcal{O})\) changes to
\[
[\{f'_{\alpha\beta}\}]=[\{c_{\alpha\beta}+f_{\alpha\beta}\}]\in H^1(S,\mathcal{O}),
\]
by an arbitrary element \([\{c_{\alpha\beta}\}]\in H^1(S,\C)\).
So the class \([\{f_{\alpha\beta}\}]\in H^1(S,\mathcal{O})\) is well defined up to adding any \([\{c_{\alpha\beta}\}]\in H^1(S,\C)\).
So the possible choices of \([\{f_{\alpha\beta}\}]\) are precisely the preimage of \(\cls\mu\in H^1(S,\mathcal{O})/\iota_*H^1(S,\C)\).

Fix the data \(\{\nu_\alpha\}\): we don't just have a closed holomorphic \(2\)-form \(\mu\) but also a choice of holomorphic \(1\)-form \(\nu_\alpha\) on each \(S_\alpha\) with \(d\nu_\alpha=\mu\).
By adjusting the functions \(f_\alpha\), we may assume that \(\nu'_\alpha=\nu_\alpha\).
With this particular normalisation, we can choose \(c_{\alpha\beta}\) above arbitrarily, subject only to the cocycle condition.
We construct the new data \(\set{\nu'_\alpha},\set{f'_{\alpha\beta}}\) by:
\[
\nu'_\alpha=\nu_\alpha, 	\quad f'_{\alpha\beta}=c_{\alpha\beta}+f_{\alpha\beta}.
\] 
We are now only free to add differences \(c_\beta-c_\alpha\) of constants to \(c_{\alpha\beta}\).
In other words, the difference \([\{c_{\alpha\beta}\}]\in H^1(S,\C)\) becomes well defined.
If the local potentials \(\{\nu_\alpha\}\) are not fixed, only \(\iota_*[\{c_{\alpha\beta}\}]\) is well defined.

So every closed holomorphic \(2\)-form \(\mu\) is represented by local data \((\{\nu_\alpha\}\), \(\set{f_{\alpha\beta}})\) in the sense that
\[
\mu=d\nu_\alpha, df_{\alpha\beta}=\nu_\beta-\nu_\alpha.
\]
Fixing the potentials \(\set{\nu_\alpha}\), the data \(\set{f_{\alpha\beta}}\) is uniquely determined, up to adding any element of \(H^1(S,\C)\).
Conversely, any element of \(H^1(S,\C)\) arises uniquely in this way.

The $2$-form $\mu$ is holomorphically exact just when \(\mu=d\nu\) for some global holomorphic $1$-form \(\nu\) on $M$.
But then \(\nu-\nu_\alpha\) is exact on \(S_\alpha\), say \(\nu-\nu_\alpha=df_\alpha\).
Conversely, if we can alter every $\nu_\alpha$, by adding some closed, hence holomorphically exact, holomorphic form $df_\alpha$ so that $\nu_\alpha+df_\alpha=\nu_\beta+df_\beta$, then $\nu:=\nu_\alpha+df_\alpha$ on $S_\alpha$ is a global holomorphic \(1\)-form with $d\nu=\mu$. 

These $df_\alpha$ then satisfy 
\[
\nu_\beta-\nu_\alpha=df_\alpha-df_\beta=d(f_\alpha-f_\beta),
\]
that is
\[
c_{\alpha\beta}:=f_{\alpha\beta}+f_\beta-f_\alpha
\]
are constants.
So $\mu$ is holomorphically exact precisely when $f_{\alpha\beta}$ is the image of some $c_{\alpha\beta}$ in cohomology, \ie just when \(\cls\mu=0\).

\subsection{Symplectic exactness}
\begin{definition}
Let $S$ be a complex manifold  of complex dimension $2N$  for some $N\geq 1$.
A \emph{holomorphic symplectic form} on $S$ is a closed holomorphic $2$-form $\omega$ on $S$ so that $\omega^N$ is nowhere zero.
A holomorphic symplectic form $\omega$ on \(S$ is \emph{holomorphically ($C^{\infty}$)  exact} if $\omega=d\vartheta$ for some holomorphic ($C^{\infty}$) complex valued $1$-form $\vartheta$.
We then say that $(S,\omega)$ is a \emph{holomorphically ($C^{\infty}$) exact symplectic manifold}.
A contact symplectic lift is \emph{holomorphically ($C^{\infty}$) exact} when its underlying holomorphic symplectic manifold is.
\end{definition}
\begin{remark}
We find numerous equivalent conditions for holomorphic exactness in section~\ref{subsection:monodromy}; the key obstruction is a monodromy class in $H^1(S,\C)$.
\end{remark}
\begin{definition}\label{def:adapted}
An \emph{adapted cover} of a contact symplectic lift $(M, V, v, S, \omega)$ is a good Leray open cover \(\{S_\alpha\}\) by contractible open sets together with contact symplectic coordinates \((z_\alpha,w_\alpha,y_\alpha)\) on each \(M_\alpha:=\pi^{-1}S_\alpha\).
We always let
\begin{align*}
S_{\alpha\beta}&:=S_\alpha\cap S_\beta,\\
S_{\alpha\beta\gamma}&:=S_\alpha\cap S_\beta\cap S_\gamma.
\end{align*}
We let
\[
\nu_\alpha:=-\sum_{j=1}^N w_{\alpha j}dz_{\alpha j}.
\]
Then
\begin{equation}\label{Eq:form-Darboux}
\left.\omega\right|_{S_\alpha}=\sum_{j=1}^N d z_{\alpha j}\wedge d w_{\alpha j}=d\nu_\alpha.
\end{equation}
Note that on each \(S_{\alpha\beta}\), \(\nu_\beta-\nu_\alpha\) is exact, say \(\nu_\beta-\nu_\alpha=df_{\alpha\beta}\).
Moreover, on each \(S_{\alpha\beta\gamma}\),
\[
c_{\alpha\beta\gamma}:=f_{\alpha\beta}+f_{\beta\gamma}+f_{\gamma\alpha}
\]
is constant.
\end{definition}
\begin{theorem}\label{thm:C.infty.exact}
Every contact symplectic lift is \(C^\infty\)-exact.
\end{theorem}
\begin{proof}
We present two proofs.

\emph{First proof: differential topology.}
As $\C$ is contractible, every smooth principal $\C$-bundle admits a global smooth section $Y$ \cite[Vol.~1, Chapter~1, \S5, p.~58, Theorem~5.7]{KNI}.
Denoting the projection map by \(\pi\), \(\pi\circ Y\) is the identity on \(S\), so \(Y^*\pi^*\) is the identity on \(\Omega^*_S\), so
\[
\omega=Y^*\pi^*\omega=Y^*d\xi=d(Y^*\xi)
\]
is $C^\infty$-exact.
%

\emph{Second proof: \u{C}ech cohomology.}
We reuse this argument below.
Choose an adapted cover as in Definition~\vref{def:adapted}.
Let
\begin{equation}\label{Eq:local-contact}
\xi_\alpha:=\left.\xi\right|_{M_\alpha}=dy_\alpha+\nu_\alpha.
\end{equation}
The transition functions of the $\C$-bundle $M\to S$ are holomorphic functions $f_{\alpha\beta}$ on $S_{\alpha\beta}$ defined by
\begin{equation}\label{Eq:change-C-bundle-f}
f_{\alpha\beta}:=y_\alpha-y_\beta.
\end{equation}
Applying \ref{Eq:local-contact}, 
\begin{equation}\label{Eq:M-to-cocycle}
df_{\alpha\beta}=dy_\alpha-dy_\beta=(\xi_\alpha-\nu_\alpha)-(\xi_\beta-\nu_\beta)=\nu_\beta-\nu_\alpha.
\end{equation}
In particular, \(f_{\alpha\beta}\) is constant on the fibres of \(M\to S\), so is a function on \(S_{\alpha\beta}\) with
\begin{equation}\label{Eq:cocy-princ-bd-zero}
f_{\alpha\beta}+f_{\beta\gamma}+f_{\gamma\alpha}=0 
\end{equation}
on \(S_{\alpha\beta\gamma}\).
As in section~\vref{section:De-Rham-coh},
the class of 
\((f_{\alpha\beta}+f_{\beta\gamma}+f_{\gamma\alpha})|_{S_{\alpha\beta\gamma}}\)
is the image of \(\omega\) in $H^2(S,\C)$. By \eqref{Eq:cocy-princ-bd-zero}, this class vanishes, so $\omega$ is $C^\infty$-exact. 
The cohomology vanishing consists precisely in the existence of \(C^\infty\) functions \(Y_{\alpha}\) on each overlap \(S_{\alpha\beta}\) so that \(f_{\alpha\beta}=Y_{\beta}-Y_{\alpha}\).
We then have a section \(Y\) defined by \(y_{\alpha}=Y_{\alpha}\).
\end{proof}
Theorem~\ref{thm:C.infty.exact} shows that all contact symplectic lifts are \(C^{\infty}\)-exact.
By contrast, holomorphic exactness is a finer condition, controlled by the cohomology class \(\cls\omega\).
\begin{corollary}
Suppose that \((S,\omega)\) is a holomorphic symplectic manifold.
Suppose that it admits a contact symplectic lift.
Suppose that \(X\) is a compact reduced complex space of even complex dimension \(2k\).
Every holomorphic map \(X\xrightarrow{\varphi}S\) satisfies \(\varphi^*\omega^k=0\).
Moreover, $S$ is not isomorphic to the smooth locus of a compact reduced complex space.
\end{corollary}
\begin{proof}
By Theorem~\vref{thm:C.infty.exact}, we can suppose that \(\omega=d\nu\) for a \(C^\infty\) \(1\)-form \(\nu\) on $S$.
Let
\[
C:=\frac{(-1)^k}{4^k (k!)^2},\qquad\Omega:=C\omega^k\wedge\bar\omega^k.
\]
The form $\Omega$ is real and semipositive: take a smooth point \(x_0\in X\) near which \(X\xrightarrow{\varphi}S\) is an immersion, so
\[
\varphi^*\omega^k=f(z)dz_1\wedge\dots\wedge dz_{2k},
\]
in local holomorphic coordinates on $X$:
\[
\varphi^*\Omega=|f(z)|^2 dx_1\wedge dy_1\wedge\dots\wedge dx_k\wedge dy_k.
\]
The singularities of \(X\) are of real codimension two or more.
Let \(X_s\subseteq X\) be the set of smooth points.
So \(\int_{X_s}\Omega=\int_X\Omega\) \cite[p.~33]{GH78}.
So \(\int_X\varphi^*\Omega>0\) or else \(\varphi^*\omega^k=0\) at every smooth immersed point of \(X\), so at the generic point, so \(\varphi^*\omega^k=0\) on \(X\).

But \(\Omega\) is exact
\[
\Omega=Cd(\nu\wedge \omega^{k-1}\wedge\bar\omega^k).
\]
By Stokes's theorem for analytic varieties \cite[p.~33]{GH78}, \(\int_X\Omega=0\).
Hence \(\varphi^*\omega^k=0\) on \(X\).

Finally, $\omega$ is a holomorphic symplectic form on $S$, so $\int_S\Omega>0$ so $S$ is not the smooth locus of any compact reduced complex space.
\end{proof}
\begin{theorem}[Existence of lift]\label{thm:C.infty.exact.iff.csl}
A holomorphic symplectic manifold is \(C^\infty\)-exact if and only if it belongs to a contact symplectic lift.
\end{theorem}
Foreman \cite[Theorem~4.3]{For00} has a similar result and proof, assuming compactness of base and total space.
For real contact structures, Boothby and Wang \cite{BW} have a similar result and proof. 
\begin{proof}
By Theorem~\vref{thm:C.infty.exact}, if a holomorphic symplectic manifold \((S,\omega)\) has a contact symplectic lift then it is $C^{\infty}$-exact.

Suppose that \((S,\omega)\) is \(C^\infty\)-exact.
Select an adapted cover as in Definition~\vref{def:adapted}.
On each $S_{\alpha\beta}$, the $1$-form $\nu_\beta-\nu_\alpha$ is closed, hence holomorphically exact, so admits a holomorphic primitive 
\begin{equation}\label{Eq:co-omega-cic}
df_{\alpha\beta}=\nu_\beta-\nu_\alpha.
\end{equation} 
As in section~\vref{section:De-Rham-coh},
$c_{\alpha\beta\gamma}:=f_{\alpha\beta}+f_{\beta\gamma}+f_{\gamma\alpha}$ is constant on \(S_{\alpha\beta\gamma}\), representing $[\omega]=0\in H^2_{dR}(S)$. 
So there are constants $c_{\alpha\beta}$ so that $c_{\alpha\beta\gamma}=c_{\alpha\beta}+c_{\beta\gamma}+c_{\gamma\alpha}$. Replacing $f_{\alpha\beta}$ with $f_{\alpha\beta}-c_{\alpha\beta}$, $[\{f_{\alpha\beta}\}]\in H^1(S,\mathcal O)$ are the transition maps of a $\C$-principal bundle 
\[
\begin{tikzcd}
\C\arrow{r}&M\arrow{d}\\
&S,
\end{tikzcd}
\]
by \eqref{Eq:change-C-bundle-f}.
Define the contact form $\xi_\alpha=dy_\alpha+\nu_\alpha$ on $S_\alpha\times \C$ as in \eqref{Eq:local-contact}. 
On $S_{\alpha\beta}$, by \eqref{Eq:co-omega-cic},
\[
\xi_\beta-\xi_\alpha=dy_\beta+\nu_{\beta}-(dy_\alpha+\nu_\alpha)
=df_{\alpha\beta}-\nu_\beta+\nu_\alpha=0.
\]
Therefore, the $\xi_\alpha$ glue to a global contact form $\xi$ on $M$. Its Reeb vector field $v$ on $S_\alpha\times\C$ is $\partial_{y_\alpha}$, hence $\C$-complete and with flow acting freely on $M$. Therefore, $(M, V, v, S, \omega)$ is a contact symplectic lift. 
\end{proof}

\subsection{Holomorphic flat connections}
We review standard material on holomorphic connections \cite{Atiyah57,Huy05} in a form useful for contact symplectic lifts.

Take a holomorphic left principal bundle 
\[
\begin{tikzcd}
G\arrow{r}&P\arrow[d,"\pi"]\\
&X,
\end{tikzcd}
\]
on a complex manifold \(X\), with action denoted $(g,p)\mapsto gp$.
Every element \(A\) in the Lie algebra \(\LieG\) of $G$ has fundamental vector field
\[
\mathscr{X}_A(p):=\left.\frac{d}{dt}\right|_{t=0}e^{tA}p.
\]
These vector fields constitute a Lie algebra action.
Let $\LieG_P:=P\times\LieG$ be the trivial vector bundle over $P$ with fibre $\LieG$.
Map
\[
(p,A)\in \LieG_P\xmapsto{\mathscr{X}}\mathscr{X}_A(p)\in T_p P.
\]
This map fits into the exact sequence
\[
\begin{tikzcd}
0\arrow{r}&\LieG_P\arrow[r,"\mathscr{X}"]&TP\arrow[r,"d\pi"]&\pi^*TM\to 0,
\end{tikzcd}
\]
of holomorphic vector bundles on \(P\).
A \emph{holomorphic connection} $V$ on this bundle is a \(G\)-invariant holomorphic splitting $TP=\LieG_P\oplus V$, \ie a \(G\)-invariant holomorphic distribution complementary to the vertical vector fields.

A \emph{holomorphic connection form} is a \(\LieG\)-valued holomorphic \(1\)-form \(\xi\) on \(P\) so that \(\xi(\mathscr{X}_A(p))=A\) for all \(A\in\LieG\) and \(g^*\xi=\Ad_g\xi\) for any \(g\in G\).
Then $V=\ker \xi$ is a connection.
Conversely every connection \(V\) has a unique connection form \(\xi\).
The curvature of the connection is the \(\LieG\)-valued \(2\)-form \(\Omega\) so that, for any tangent vectors \(u,w\in T_p P\),
\[
\Omega(u,w):=d\xi(u,w)+[\xi(u),\xi(w)].
\]

For $G=\C$ (as opposed to $\C^*$, the structure group of line bundles), we denote the bundle by
\[
\begin{tikzcd}
\C\arrow{r}&M\arrow{d}\\
&S.
\end{tikzcd}
\]
Let \(v:=\mathscr{X}_1\).
A complex manifold $M$ with a $\C$-complete holomorphic vector field whose flow acts freely and properly is precisely a principal $\C$-bundle over $S=M/\C$.
As $\C$ is abelian, the curvature is $d\xi$.
The curvature is the pullback of a unique holomorphic $2$-form $\omega$ on $S$ by Lemma~\vref{lemma:semibasic}.
A holomorphic connection $V$ on $M$ is precisely a $v$-invariant holomorphic hyperplane field $V$ transverse to $v$.
The connection $1$-form $\xi$ is the unique $1$-form with $\xi(v)=1$ and $\left.\xi\right|_V=0$, automatically $v$-invariant.

Summing up:
\begin{lemma}\label{lemma:quotient.by.vf}
Suppose that \(M\) is a complex manifold.
Suppose that \(v\) is a \(\C\)-complete holomorphic vector field on \(M\), whose flow acts freely and properly.
The quotient \(S\) by that action makes 
\[
\begin{tikzcd}
\C\arrow{r}&M\arrow[d,"\pi"]\\
&S
\end{tikzcd}
\]
into a holomorphic principal \(\C\)-bundle.
Every choice of holomorphic hyperplane field \(V\subseteq TM\) with \(\LieDer_v V\subseteq V\) and \(v(p)\not\in V_p\) for every $p\in M$ is a holomorphic connection for that bundle.
Its connection form is the unique \(1\)-form \(\xi\) with \(\xi(v)=1\) and \(\xi=0\) on \(V\).
It is \(v\)-invariant and holomorphic. 
Conversely, every \(1\)-form \(\xi\) with \(\LieDer_v\xi=0\) and \(\xi(v)=1\) is the connection form for a unique holomorphic connection \(V=\ker\xi\).
The curvature of that connection is \(d\xi\).
The curvature is the pullback of a unique holomorphic $2$-form $\omega$ on $S$.
\end{lemma}
Suppose in addition that $M$ has odd complex dimension.
Then $(M,V,v)$ is a Reeb manifold if and only if the curvature $\omega$ is \emph{non-degenerate} in the sense that $(d\xi)^N\wedge\xi\ne 0$.
We characterise the contact symplectic lifts that are holomorphically exact in terms of flat connections.
\begin{proposition}\label{prop:flat-conn}
The contact form of any contact symplectic lift is a holomorphic connection form for its bundle.
The curvature of that connection is the symplectic form on the symplectic base manifold.
A contact symplectic lift \(M\to S\) is holomorphically exact if and only if its principal \(\C\)-bundle admits a holomorphic flat connection, or equivalently, the pullback bundle 
\[
\begin{tikzcd}
\tilde M\arrow{r}\arrow{d}&M\arrow{d}\\
\tilde{S}\arrow{r}&S
\end{tikzcd}
\]
to some covering space \(\tilde S\to S\) is a holomorphically trivial principal \(\C\)-bundle.
%
%
\end{proposition}
\begin{proof}
\par\noindent%
\begin{ProofSteps}
\ProofStep{Holomorphic exactness \(\iff\) a flat connection}
By Lemma~\vref{lemma:quotient.by.vf}, a holomorphic connection for the bundle $M\to S$ is precisely a holomorphic $1$-form $\xi'$ on $M$ so that $\ip_v{\xi'}=1$ and $\LieDer_v \xi'=0$, with curvature pulling back to $d\xi'$.

Suppose $\xi'$ is flat.
Let $\eta:=\xi-\xi'$ so $\ip_v{\eta}=0$.
Then $\eta$ is invariant under the flow of $v$ so $\LieDer_v\eta=0$.
By the Cartan formula,
\[
\ip_v{d\eta}=\LieDer_v\eta-d\ip_v{\eta}=0-0=0.
\]
By Lemma~\vref{lemma:semibasic}, $\eta=\pi^*\nu$ for a unique holomorphic $1$-form $\nu$ on $S$.
But $d\eta=d\xi-d\xi'=\pi^*\omega$, so $\pi^*(\omega-d\nu)=0$.
Pullback of differential forms by submersions is injective, so
$\omega=d\nu$.
Conversely, if $\omega=d\nu$, let $\eta:=\pi^*\nu$, let $\xi':=\xi-\eta$.
We see that $\xi'$ is a flat holomorphic connection.

\ProofStep{A flat holomorphic connection \(\iff\) trivialises on a cover}
Every (holomorphic) flat connection pulls back to some covering space to be (holomorphically) trivial, \ie every (holomorphic) flat connection is equivalent to a representation of the fundamental group \cite[p.~200 Proposition~14]{Atiyah57}.
\end{ProofSteps}
\end{proof}
\subsection{Monodromy}%
\label{subsection:monodromy}
Take a contact symplectic lift $(M, V, v, S, \omega)$.
Suppose the holomorphic symplectic form $\omega$ on $S$ is holomorphically exact, say $\omega=d\nu$.
By Lemma~\vref{lemma:quotient.by.vf}, the $1$-form $\theta:=\xi-\pi^*\nu$ is a flat holomorphic connection.
Its monodromy around loops in $S$ is therefore a group morphism
\[
\pi_1(S)\to\C,
\]
so that $M=\tilde{S}\times^{\pi_1(M)}\C$.
In particular, the monodromy of a flat connection vanishes if and only if the bundle is holomorphically trivial.

As $\C$ is abelian, this group morphism factors through the abelianisation $H_1(S,\mathbb{Z})$ \cite[p.~225, Theorem~17.21]{Bo-Tu}.
By de~Rham duality, the monodromy determines and is determined by a unique element $\mu=\mu_{\nu}\in H^1(S,\C)$, the \emph{monodromy} of the symplectic potential $\nu$.

Concretely, pick a point $p_0\in M$ and let $x_0:=\pi(p_0)\in S$.
Each loop \(\gamma\) on \(S\) starting and ending at $x_0$ has a unique parallel lift $\hat\gamma$ on $M$ starting at $p_0$.
Here, parallel means $\theta=0$.
This path ends at some point \(p_1=e^{tv}p_0\)
where \(t=\int_{\gamma}\mu\).

Replacing $\nu$ by $\nu':=\nu-df$ changes $\theta$ to $\theta':=\xi-\pi^*\nu+d\pi^*f$.
The $\theta'$-parallel lift of a path is
\[
\hat\gamma'(t):=e^{f(\gamma(t))v}\hat\gamma(t).
\]
The monodromy is unchanged, depending only on the holomorphic de~Rham cohomology class of the holomorphic symplectic potential $\nu$.
%
Lift up every path in $S$ through $x_0$ to a path in $M$ through $p_0$, parallel for the flat connection.
By the Frobenius theorem, the lifts all lie on a unique complex hypersurface $\hat S\subseteq M$.
By $\C$-invariance, the map $\hat S\to S$ is evenly covered, so a covering map.
The monodromy vanishes just when loops lift to loops, \ie when $\hat S\to S$ is a biholomorphism, so a global holomorphic section.
The pullback contact symplectic lift on $\hat S$ has holomorphically trivial bundle. 
A covering space of $S$ has holomorphically trivial pullback bundle if and only if the covering space covers $\hat S$.


\subsection{Characterisations of holomorphic exactness}
We now prove the equivalence of various geometric conditions on a contact symplectic lift $(M,V,v,S,\omega)$, some depending only on $(M,v)$ and others depending only on $(S,\omega)$.
\begin{theorem}\label{theorem:fit}
Suppose that $(M,V,v,S,\omega)$ is a contact symplectic lift.
Then the following are equivalent:
\begin{enumerate}[label*=\arabic*.]
\item
$M$ contains a complex submanifold of codimension one intersecting each Reeb orbit once transversely.
\item
The holomorphic principal $\C$-bundle $M\to S$ admits a global holomorphic section.
\item
The holomorphic principal $\C$-bundle $M\to S$ is holomorphically trivial.
\item
The contact symplectic lift $(M,V,v,S,\omega)$ is fit, \ie \(M=S\times\C_y\) and \(\xi=dy+\pi^*\nu\) for some holomorphic \(\nu\).
\item
The symplectic $\omega$ is holomorphically symplectically exact and has trivial monodromy.
\end{enumerate}
\end{theorem}
\begin{proof}
\par\noindent%
\begin{ProofSteps}
\item[1\(\,\leftrightarrow\,\)2] 
These cases are equivalent; we include them separately only to highlight that case 1 depends only on $(M,v)$.
\item[3\(\,\rightarrow\,\)2] is clear.
\item[2\(\,\rightarrow\,\)3]
Suppose that there is a global holomorphic section $f$ of $M\to S$.
Denote by \(\Phi_y\) the Reeb flow for complex time \(y\).
The map 
\[
(x,y)\in S\times\C\xmapsto{\tau} \Phi_y(f(x))\in M,
\]
is clearly a holomorphic bijection, so a biholomorphism.
This biholomorphism intertwines translation in $y$ with the flow of $v$.
\item[3\(\,\rightarrow\,\)4] 
Suppose $M=S\times\C$ with $v=\partial_y$ and $\pi(x,y)=x$.
Let $\theta:=\xi-dy$, so $\ip_v\xi=1$, so $\ip_{\partial_y}\theta=0$.
Both $dy$ and $\xi$ are invariant under the Reeb vector field, so $\theta$ is also, so $\LieDer_{\partial y}\theta=0$.
The fibres are connected.
By Lemma~\vref{lemma:semibasic}, $\theta$ is pulled back, \ie $\theta=\pi^*\nu$ for a unique holomorphic $1$-form $\nu$ and $\pi^*d\nu=d\theta=d\xi=\pi^*\omega$ yields $d\nu=\omega$.

The $1$-form $\xi-\pi^*\nu$ is a flat connection. 
The bundle is holomorphically trivial, so the monodromy vanishes.
So the bundle admits a parallel global section, say $y=h(x)$.
Replace $y$ by $y-h(x)$ as our trivialisation of $M$.
The level sets of $y$ parallel, \ie $\xi-\pi^*\nu$ vanishes on them.
So $dy$ and $\xi-\pi^*\nu$ equal $1$ on $\partial_y$ and vanish on the level sets of $y$, so are equal:\(\xi=dy+\pi^*\nu\), fit.
\item[4\(\,\rightarrow\,\)5] 
A $\nu$-fit lift has $\omega=d\nu$ and trivial bundle so the monodromy of any flat connection vanishes: so $\mu_\nu=0$.
\item[5\(\,\rightarrow\,\)3] 
Vanishing monodromy makes the bundle trivial.
\end{ProofSteps}
\end{proof}
\begin{corollary}
Suppose that $(S,\omega)$ is a holomorphically exact holomorphic symplectic manifold.
Suppose that $H^1(S,\C)=0$.
Then every contact symplectic lift over $(S,\omega)$ is fit.
\end{corollary}
\begin{proof}
Because $(S,\omega)$ is holomorphically symplectically exact, $\omega=d\nu$ for some holomorphic symplectic potential $\nu$.
Take a contact symplectic lift $(M,V,v,S,\omega)$.
The monodromy $\mu_\nu\in H^1(S,\C)=\{0\}$ vanishes.
By Theorem~\vref{theorem:fit}, $(M,V,v,S,\omega)$ is fit.
\end{proof}
\begin{corollary}
Take a holomorphically exact contact symplectic lift $(M,V,v,S,\omega)$.
Take the covering space $\hat{S}=\tilde{S}/\mu_\nu$ given by quotienting the universal covering by the monodromy.
Then $(M,V,v,S,\omega)$ pulls back to $\hat{S}$ to become fit.
Conversely, if the pullback of the contact symplectic lift to a covering space of $S$ is fit then the pullback of $\omega$ to that covering space is holomorphically exact, say with symplectic potential $\nu$.
The monodromy of that symplectic potential vanishes on that covering space.
Equivalently, the covering space is a covering space of \(\hat{S}\).
\end{corollary}

\section{The lifts of a given symplectic manifold}\label{Sec:ineq-lifts}
In this section, we classify all contact symplectic lifts of a given holomorphic symplectic manifold, in terms of cohomological data on that manifold. The key result is that isomorphism classes of lifts are parametrised by a first cohomology group.
\subsection{Classifying the bundles that arise}
First, we find precisely which holomorphic principal \(\C\)-bundles arise as the bundles of contact symplectic lifts.

A \emph{torsor} \(X\) modelled on a group \(G\) is a set acted on freely and transitively by \(G\) \cite[p.~6]{Go15}.
Suppose that \(V\) is a vector space. 
Think of \(V\) as an abelian group under vector addition.
A torsor modelled on \(V\) is called an \emph{affine space} modelled on \(V\) \cite[p.~6]{Go15}.
We introduce this terminology because our moduli spaces below are affine spaces.
Recall that every holomorphic principal \(\C\)-bundle \(M\to S\) is identified up to isomorphism with its bundle class \([M]\in H^1(S,\mathcal{O})\).
Hence the isomorphism classes form an affine space modelled on \(H^1(S,\mathcal{O})\).
Since the trivial bundle is a distinguished point, this affine space is a vector space.
\begin{theorem}%
[Bundle moduli of contact symplectic lifts]%
\label{thm:bundle.moduli}
Suppose that \((S,\omega)\) is a \(C^{\infty}\)-exact holomorphic symplectic manifold.
Any class in \(H^1(S,\mathcal{O})\) arises as the bundle class of a contact symplectic lift \(M\to S\) just when it maps to \(\cls\omega\) in \(H^1(S,\mathcal{O})/\iota_*H^1(S,\C)\).
Any two contact symplectic lifts of the same \((S,\omega)\) have bundles whose associated bundle classes in \(H^1(S,\mathcal{O})\) differ by a unique element of \(\iota_*H^1(S,\C)\).
This difference vanishes precisely when the bundles are isomorphic.
Thus, the set of isomorphism classes of these bundles is an affine space modelled on \(\iota_*H^1(S,\C)\).
\end{theorem}
\begin{proof}
Suppose that we replace the choices  of \(\nu_\alpha\) and \(f_{\alpha\beta}\) in the proof of theorem~\vref{thm:C.infty.exact.iff.csl}, say to new choices \(\nu'_\alpha\) and \(f'_{\alpha\beta}\), to produce a new lift, with
\begin{align*}
d\nu'_\alpha&=\omega,\\
df'_{\alpha\beta}&=\nu'_\beta-\nu'_\alpha,\\
0&=f'_{\alpha\beta}+f'_{\beta\gamma}+f'_{\gamma\alpha},\\
\xi'&=dy'_{\alpha}+\nu'_\alpha.
\end{align*}
Then \(\nu'_\alpha=\nu_\alpha+df_\alpha\), for a holomorphic function \(f_\alpha\), unique up to adding a constant.
Let
\[
c_{\alpha\beta}:=f'_{\alpha\beta}-f_{\alpha\beta}-f_\beta+f_\alpha.
\]
Differentiate to find that each \(c_{\alpha\beta}\) is constant.
Check that the \(\{c_{\alpha\beta}\}\) satisfy the cocycle condition.
So the new data arises from the old by
\begin{align*}
f'_{\alpha\beta}&=c_{\alpha\beta}+f_{\alpha\beta}+f_\beta-f_\alpha,\\
\nu'_\alpha&=\nu_\alpha+df_\alpha.
\end{align*}

On the other hand, we take any constants \(c_{\alpha\beta}\) satisfying the cocycle condition, and any holomorphic function \(f_\alpha\) on each \(S_\alpha\).
These equations define a new contact symplectic lift.
Every element of \(H^1(S,\mathcal{O})\) mapping to \(\cls\omega\in H^1(S,\mathcal{O})/\iota_*H^1(S,\C)\) is the class of a holomorphic principal \(\C\)-bundle that arises as a contact symplectic lift.
%
%
\end{proof}
\begin{remark}
In the notation of the previous proof, given a closed holomorphic \(1\)-form \(\eta\) on \(S\), write \(\eta=dg_\alpha\) on each \(S_\alpha\).
Let \(c_{\alpha\beta}=g_\beta-g_\alpha\) on \(S_{\alpha\beta}\).
Check that these are constants.
Let
\begin{align*}
y'_\alpha&=y_\alpha-g_\alpha,\\
f'_{\alpha\beta}&=c_{\alpha\beta}+f_{\alpha\beta},\\
\nu'_{\alpha}&=\nu_\alpha+dg_\alpha.
\end{align*}
This is a bundle automorphism \(F\) so that \(F^*\xi=\xi+\pi^*\eta\).
So pulling back holomorphic \(1\)-forms and adding them to the contact form of a contact symplectic lift creates an isomorphic contact symplectic lift.
\end{remark}
\subsection{The moduli space}
\begin{definition}
Take a \(C^{\infty}\)-exact holomorphic symplectic manifold \((S,\omega)\).
Fix a contact symplectic lift \(\mathcal M\) of \((S,\omega)\), with local potentials \(\nu_\alpha\) and transition maps \(f_{\alpha\beta}\).
Every contact symplectic lift \(\mathcal M'\), using the same set of local potentials, has transition maps \(f'_{\alpha\beta}=f_{\alpha\beta}+c_{\alpha\beta}\) for a cocycle \(\{c_{\alpha\beta}\}\) of complex constants.
These determine a cohomology class, the \emph{difference class} of the two contact symplectic lifts, which we denote 
\[
[\mathcal M'-\mathcal M]:=[\{c_{\alpha\beta}\}]\in H^1(S,\C).
\]
\end{definition}
We prove that the difference class characterises contact symplectic lifts up to isomorphism.
The moduli space of contact symplectic lifts of a given holomorphic symplectic manifold \((S,\omega)\) is an affine space modelled on \(H^1(S,\C)\).
\begin{theorem}[Moduli of contact symplectic lifts]\label{thm:how-many-lifts}
Fix a contact symplectic lift \(\mathcal M\) of a \(C^{\infty}\)-exact holomorphic symplectic manifold \((S,\omega)\).
Every contact symplectic lift \(\mathcal M'\) of \((S,\omega)\) is determined uniquely up to isomorphism by the difference class \([\mathcal M'-\mathcal M]\).
Every cohomology class in \(H^1(S,\C)\) occurs as this difference class.
Thus, the space of  isomorphism classes of contact symplectic lifts of \((S,\omega)\) is an affine space modelled on \(H^1(S,\C)\).
\end{theorem}
\begin{proof}
Continuing from the proof of theorem~\vref{thm:bundle.moduli}, every bundle isomorphism on the new bundle is expressed precisely as
\[
y''_\alpha:=y'_\alpha+g_\alpha,
\]
for any holomorphic function \(g_\alpha\) on each \(S_\alpha\).
This bundle isomorphism changes the transition functions to
\[
f''_{\alpha\beta}=f'_{\alpha\beta}+g_\alpha-g_\beta.
\]
So if we choose \(g_\alpha=f_\alpha\), we arrange that the new contact symplectic lift arises from the old lift by:
\begin{align*}
f'_{\alpha\beta}&=c_{\alpha\beta}+f_{\alpha\beta},\\
\nu'_\alpha&=\nu_\alpha.
\end{align*}
So in a suitable choice of local trivialisations, we can arrange that the new and old contact symplectic lifts are expressed precisely by the differences \(c_{\alpha\beta}\).
These are then uniquely determined up to adding differences \(c_\beta-c_\alpha\) of constants.
\end{proof}

\begin{corollary}\label{Cor:unique-contact}
Let $S$ be a complex manifold of complex dimension $2N$. Then $H^1(S, \C)=\{0\}$ just when every $C^\infty$-exact holomorphic symplectic form $\omega$ on $S$ has a unique (up to isomorphism) contact symplectic lift.
Furthermore, if $\omega$ is a holomorphically exact symplectic form on $M$, say $\omega=d\nu$ for some holomorphic $1$-form on $S$, then, up to isomorphism, the only contact symplectic lift of $(S,\omega)$ is $(S\times \C, \ker \xi, \partial_y, S, d\nu)$ with contact form $\xi=dy+\nu$ and $y\in \C$.
\end{corollary}

Take two contact symplectic lifts $(M,V,v,S,\omega)$ and $(M',V',v',S,\omega)$.
It might be that $M\to S$ and  $M'\to S$ are isomorphic as $\C$-principal bundles but not isomorphic as contact symplectic lifts.
This can occur even if $S$ is Kobayashi hyperbolic.
\begin{example}%
\label{non-equiv}
Let $S:=\D^\ast\times\D$, where $\D^\ast$ denotes the punctured disc. Then $S$ is Kobayashi hyperbolic, Stein, and $H^1(S,\C)=\C$. Denote points of $\C^2$ by $(z,w)$.
Let $\omega:=dz\wedge dw$. Then $(S,\omega)$ is holomorphically exact. By Theorem~\ref{thm:how-many-lifts}, contact symplectic lifts of $(S,\omega)$ are in bijection with $H^1(S,\C)=\C$.
In particular, there are infinitely many lifts, all contact-hyperbolic.
As $S$ is Stein, $H^1(S,\mathcal O_S)=0$, so every holomorphic principal $\C$-bundle is trivial, so the lifts of \((S,\omega)\) are contact structures on the trivial bundle \(M=S\times\C\).

To be more explicit, consider the Leray cover \(\{S_0,S_1,S_2\}\) where
\[
S_j = \set[(z,w) \in S]{\Arg(z) \in \left(\frac{2\pi j}{3}-\varepsilon, \frac{2\pi(j+1)}{3} + \varepsilon\right)},
\]
and \(\varepsilon\) can be any real number with \(0 < \varepsilon < \pi/3\).
Any class in $H^1(\{S_\alpha\}, \C)$ has a representative of the form $(0,0,c)$ for $c \in \C$. Choose a branch of $\log z$ on each $S_j$ that $y_j:=y + \frac{c}{2 \pi i} \log z$ satisfy $y_1 - y_2 = 0$, $y_2 - y_3 = 0$, but $y_3 - y_1 = c$.
So $dy_j = dy + \frac{c}{2 \pi i} \frac{dz}{z}$.
We obtain nonisomorphic contact manifolds $(S \times \C, dy - (w - \frac{c}{2\pi iz})dz)$ for $c \in \C$.
\end{example}
\subsection{Reeb morphisms}
We describe morphisms between lifts.
\begin{definition}
Take Reeb manifolds \((M,V,v)\) and \((M',V',v')\).
Take a complex constant \(\lambda\ne 0\).
A \emph{\(\lambda\)-scale Reeb morphism} \((M,V,v)\xrightarrow{\Phi}(M',V',v')\) is a holomorphic immersion \(M\xrightarrow{\Phi}M'\)
equivariant for the \(\C\)-actions, so that \(\Phi_*v=v'/\lambda\) and \(\Phi^*\xi'=\lambda \xi\).
A \emph{scale Reeb morphism} is a \(\lambda\)-Reeb morphism for some \(\lambda\).
\end{definition}
\begin{definition}
A \emph{scale Reeb automorphism} of a Reeb manifold \((M,V,v)\) is a bijective scale Reeb morphism $\Phi\colon (M,V,v)\to (M,V,v)$.
Denote by \(\Aut (M,V,v)\) the group of scale Reeb automorphisms.
\end{definition}
\begin{definition}
Take holomorphic symplectic manifolds $(S,\omega)$ and $(S',\omega')$ and a complex number $\lambda\ne 0$.
A \emph{$\lambda$-scale symplectic map} is a holomorphic map \(S\xrightarrow{\varphi}S'\) so that \(\varphi^*\omega'=\lambda\omega\).
A \emph{scale symplectic map} is a $\lambda$-scale symplectic map for some complex constant \(\lambda\ne 0\).
\end{definition}
\begin{definition}
A \emph{scale symplectic automorphism} of a holomorphic symplectic manifold \((S,\omega)\) is a scale symplectic map $\varphi\colon (S,\omega)\to (S,\omega)$.
We also say that $\varphi$ \emph{preserves $\omega$ up to scale}.
Denote by \(\Aut_{\C\omega}S\) the group of scale symplectic automorphisms.
\end{definition}
\begin{definition}
Take contact symplectic lifts \(\mathcal M=(M,V,v,S,\omega)\) and \(\mathcal M'=(M',V',v',S',\omega')\).
Take a complex constant \(\lambda\ne 0\).
A \emph{\(\lambda\)-scale contact symplectic lift morphism} \(\mathcal M\xrightarrow{\Phi}\mathcal M'\) is a \(\lambda\)-scale Reeb morphism $(M,V,v)\xrightarrow{\Phi}(M',V',v')$ and a \(\lambda\)-scale symplectic map \(S\xrightarrow{\varphi}S'\) so that the diagram
\[
\begin{tikzcd}
M\arrow[d,"\pi"']\arrow[r,"\Phi"]&M'\arrow[d,"\pi'"]\\
S\arrow[r,"\varphi"]&S'
\end{tikzcd}
\]
commutes.
A \emph{scale contact symplectic lift morphism} is a \(\lambda\)-contact symplectic lift morphism for some \(\lambda\ne 0\).
A \emph{scale contact symplectic lift isomorphism} is a bijective scale contact symplectic lift morphism.
A \emph{scale contact symplectic lift automorphism} is a scale contact symplectic lift isomorphism \(\mathcal M\xrightarrow\varphi\mathcal M\) from a contact symplectic lift to itself.
\end{definition}
Part of the motivation for the definition of scale Reeb morphism is that we can only rescale Reeb vector fields by constants.
\begin{lemma}\label{lemma:rescaling.Reeb}
Suppose that \((M,V,v)\) is a Reeb manifold and \(f\) is a holomorphic function on \(M\).
Suppose that \(M\) is connected.
Then \(fv\) is also a Reeb vector field if and only if \(f\) is a nonzero constant.
\end{lemma}
\begin{proof}
Let \(w:=fv\).
Denote by \(\xi\) the contact form associated to \(v\).
To preserve \(V\), we need precisely that \(w\) preserves \(\xi\) up to a multiple of \(\xi\), \ie \(0=\xi\wedge \LieDer_w \xi\).
Recall
\[
0=\LieDer_v\xi=\ip_v(d\xi)+d(\ip_v\xi)=\ip_v(d\xi),
\]
so \(\xi\wedge\ip_v(d\xi)=0\).
Compute:
\begin{align*}
\xi\wedge\LieDer_w \xi
&=
\xi\wedge(\ip_w{d\xi}+d(\ip_w{\xi})),
\\
&=\xi\wedge(f\ip_v{d\xi}+df),
\\
&=\xi\wedge\wedge df.
\end{align*}
So \(\xi\wedge df=0\), say \(df=g\xi\).
Differentiate: \(0=dg\wedge \xi+g\,d\xi\).
Wedge: \(0=g\xi\wedge (d\xi)^N\).
So \(g=0\), and \(df=0\), \(f\) is locally constant.
Since \(M\) is connected, \(f\) is constant.
\end{proof}

\subsection{Pullback of contact symplectic lifts}
\begin{definition}
Suppose that $\mathcal M'=(M',V',v',S',\omega')$ is a contact symplectic lift and that $S\xrightarrow\varphi S'$ is a holomorphic map.
Let $M:=\varphi^*M'$ be the pullback bundle, with the pullback bundle morphism $\Phi$:
\[
\begin{tikzcd}
M\arrow[d,"\pi"']\arrow[r,"\Phi"]&M'\arrow[d,"\pi'"]\\
S\arrow[r,"\varphi"]&S'.
\end{tikzcd}
\]
For each point $p=(s,p')\in M$, let $V_p:=d\Phi_p^{-1}V'_{p'}\subseteq TM$.
As $M\to S$ is a principal right $\C$-bundle, differentiating the $\C$-action gives an infinitesimal generator, the unique holomorphic vector field $v$ on $M$ making the flows commute:
\[
e^{tv}(s,p')=(s,e^{tv'}p').
\]
Let \(\omega:=\varphi^*\omega'\).
Then the tuple $\mathcal M:=(M,V,v,S,\omega)$ is the \emph{pullback} of the contact symplectic lift $(M',V',v',S',\omega')$ by the map $S\xrightarrow\varphi S'$, also denoted 
\(\varphi^*\mathcal M'\).
\end{definition}

The pullback $\mathcal M$ is not necessarily a contact symplectic lift, because the pullback form $\varphi^*\omega$ is not necessarily symplectic.
\begin{proposition}\label{prop:pullback.symplectic}
The pullback $(M,V,v,S,\omega)=\varphi^*(M',V',v',S',\omega')$ of a contact symplectic lift by a holomorphic map $S\xrightarrow{\varphi}S'$ is a contact symplectic lift if and only if 
\[
(\varphi^*\omega')^{\dimC S/2}\neq 0,
\]
\ie $S\xrightarrow{\varphi}S'$ pulls back the holomorphic symplectic form on $S'$ to a holomorphic symplectic form.
If this occurs then the associated principal bundle morphism $M\xrightarrow{\Phi}M'$ is a $1$-scale Reeb morphism.
It is a $1$-scale contact symplectic lift morphism.
\end{proposition}
\begin{proof}
Let $N:=(1/2)\dimC{S}$.
If the pullback is a contact symplectic lift then $\omega$ is a holomorphic symplectic form so $\omega^N\ne 0$.

Suppose that $\omega^N\ne 0$.
As $\omega=\varphi^*\omega'$, 
\[
d\omega=d(\varphi^*\omega')=\varphi^*d\omega'=0.
\]
So $\omega$ is a holomorphic symplectic form and $S\xrightarrow\varphi S'$ is a holomorphic immersion.
Take the contact form $\xi'$ on $M'$ associated to the Reeb vector field $v'$.
Pick a point $p\in M$, say $p=(s,p')\in S\times M'$.
By Lemma~\vref{lemma:nu.fit}, we have a connected open set $U'\subseteq S'$ containing $p'$ with Darboux coordinates $z,w$, and coordinates $(z,w,y)$ on the preimage of \(U'\) in \(M'\) making $U'$ fit.
Shrinking $U'$ if needed, some open neighborhood $U\subseteq S$ of $s$ is a closed embedded symplectic submanifold of $U'$.

By the Darboux--Weinstein theorem \cite[p.~155 Theorem~22.1]{GS1990}, with statement and proof applied verbatim for holomorphic symplectic manifolds, after perhaps replacing $U'$ by a smaller open set, we have \emph{adapted} global Darboux coordinates $z,w$, \ie global Darboux coordinates in which the symplectic submanifold is
\begin{align*}
z_j&=0, j=N+1,\dots,N', \\
w_j&=0, j=N+1,\dots,N':
\end{align*}
with symplectic form
\begin{align*}
\omega'&=
dz_1\wedge dw_1+
\dots
+
dz_N\wedge dw_N+
\dots
+
dz_{N'}\wedge dw_{N'},\\
\omega&=
dz_1\wedge dw_1+
\dots
+
dz_N\wedge dw_N.
\end{align*}
and \(\xi'=dy+\nu'\) with 
\[
\nu'=-w_1\,dz_1-\dots-w_N\,dz_N-\dots-w_{N'}dz_{N'}.
\]
The open set $\pi^{-1}U\subseteq M$ is precisely cut out by the same equations as $U$, \ie 
\begin{align*}
z_j&=0, j=N+1,\dots,N', \\
w_j&=0, j=N+1,\dots,N',
\end{align*}
with $y$ free.
With $\nu=:-\sum_{j=1}^Nw_j dz_j$, \(\xi=dy+\nu\) on $\pi^{-1}U$ is the pullback of $\xi'$, and a contact form with kernel $V_p$.
So \(\xi\) vanishes precisely on the vectors in $T_p M$ that lie tangent to $\pi^{-1}U$ and lie inside $V'_{p'}$, \ie have $\xi'=0$, \ie precisely on $V_p$.
\end{proof}
Given two contact symplectic lifts \(\mathcal M,\mathcal M'\) of the same holomorphic symplectic manifold, Theorem~\vref{thm:how-many-lifts} defines a difference class \([\mathcal M-\mathcal M']\in H^1(S,\C)\).
The difference class vanishes precisely when they are isomorphic.

For a contact symplectic lift 
\[
\mathcal M:=(M,V,v,S,\omega),
\]
and complex number \(\lambda\ne 0\), let
\[
\lambda\mathcal M:=(M,V,v/\lambda,S,\lambda\omega).
\]
This scaling commutes with isomorphism to define a scaling by complex nonzero constants on isomorphism classes and 
\[
\lambda[\mathcal M'-\mathcal M]
=
[\lambda\mathcal M'-\lambda\mathcal M]\in H^1(S,\C).
\]

Take two contact symplectic lifts \(\mathcal M:=(M,V,v,S,\omega)\), \(\mathcal M':=(M',V',v',S',\omega')\), a complex constant $\lambda\ne 0$, and a $\lambda$-scale symplectic map \(S\xrightarrow{\varphi}S'\).
Proposition~\vref{prop:pullback.symplectic} says that the pullback is an immersion taking tangent spaces to symplectic subspaces.
By Proposition~\vref{prop:pullback.symplectic}, the pullback is a contact symplectic lift.
\begin{theorem}\label{thm:S.S.prime}
The difference class \([\varphi^*\mathcal M'-\lambda \mathcal M]\) vanishes 
if and only if there is a Reeb morphism
\[
\begin{tikzcd}
M\arrow[r,"\Phi"]\arrow[d,"\pi"']&M'\arrow[d,"\pi'"]\\
S\arrow[r,"\varphi"]&S'.
\end{tikzcd}
\]
If this occurs, \(\Phi\) is a \(\lambda\)-Reeb morphism.
It is unique up to composition with a constant time flow the Reeb vector on \(M\) over each component of \(S\).
\end{theorem}
\begin{proof}
Replacing \(\mathcal M\) by \(\lambda\mathcal M\), we can assume that \(\lambda=1\).
Denote by \(M''\xrightarrow\Gamma M'\) the bundle morphism of \(\mathcal M''=\varphi^*\mathcal M\).
By Theorem~\vref{thm:how-many-lifts}, since \(0=[\mathcal M''-\mathcal M]\), we  have a \(1\)-contact symplectic lift morphism.
Denote the morphism by \(M\xrightarrow\Delta M''\) over \(S\), so $\Gamma\circ\Delta$ is a $1$-Reeb morphism.
Conversely, any Reeb morphism factors through $\mathcal M''$.
Uniqueness up to constant-time Reeb flows follows from Corollary~\vref{cor:autos.id.on.S}.
\end{proof}

\section{Hyperbolicity of contact symplectic lifts}%
\label{section:hyperbolicity.of.lifts}
 In this section we prove that contact hyperbolicity of the total space of a contact symplectic lift corresponds to Kobayashi hyperbolicity of the base.

Suppose that $S$ is a complex manifold.
Denote by $\kappa_S$ its Kobayashi pseudometric and by $k_S$ its Kobayashi pseudodistance \cite{AbateTaut,Kob}.

\begin{lemma}\label{V.hyp.to.hyp}
If a contact symplectic lift is contact-hyperbolic then its symplectic base manifold is Kobayashi hyperbolic.
\end{lemma}
\begin{proof}
Take a contact-hyperbolic contact symplectic lift \((M,V,v,S,\omega)\).
We need to prove that $S$ is Kobayashi hyperbolic.
Take distinct $\bar p,\bar q\in S$ and any $p\in\pi^{-1}(\bar p)$. 
Suppose that $k_S(\bar p,\bar q)=0$.

By definition of the Kobayashi pseudodistance, and its agreement with the chain pseudodistance \cite{Ro71}, there exist chains of holomorphic discs in $S$ joining $\bar p$ to $\bar q$ with arbitrarily small length.  
By Proposition~\vref{Prop:V-lifts}, each such disc lifts to a $V$-disc in $M$.  
The lifted chain starts at $p$ and ends at a point in $\pi^{-1}(\bar q)$.
The length of the chain in $S$ is the length of its lift in $M$.

By equation~\vref{Eq:quasi-Barth}, there are $V$-paths of arbitrarily small length from $p$ to points of \(\pi^{-1}\bar q\).
As $k_V$ induces the manifold topology on $M$, points of $\pi^{-1}(\bar q)$ accumulate at $p$ in the manifold topology.  
This contradicts continuity of $\pi$.
\end{proof}
The \emph{velocity} of a holomorphic disc $\D\xrightarrow\varphi M$ is $\varphi'(0)$.
\begin{lemma}\label{lemma:kV.eq.kS}
Take a contact symplectic lift \((M,V,v,S,\omega)\) with projection $\pi$.
For each point $p\in M$, if $\bar p:=\pi(p)$, the linear isomorphism
\[
v\in V_p\xrightarrow{d\pi_p}\bar v\in T_{\bar p}S
\]
is a bijection between the set of velocities of $V$-discs and the set of velocities of discs, preserving the pseudometrics:
\begin{equation}\label{Eq:kappaSV}
\kappa_S(\bar p;\bar v)=\kappa_V(p;v).
\end{equation}
Any rectifiable $V$-path in $M$ projects to a path in $S$ of the same length:
\[
\ell_S(\bar\gamma)=\ell_V(\gamma), \qquad \bar\gamma:=\pi\circ\gamma.
\]
Projection does not expand the pseudodistances:
\[
k_S(\bar p,\bar q)\le k_V(p,q),
\]
or the chain pseudodistances:
\[
\tilde k_S(\bar p,\bar q)\le \tilde k_V(p,q),
\]
for any points \(p,q\in M\) and their images \(\bar p,\bar q\in S\).
We again see (without invoking chains of discs) that if $(M,V)$ is contact-hyperbolic then $S$ is Kobayashi hyperbolic.
\end{lemma}
\begin{proof}
As $V$ and $v$ are transverse, $T_p M=V_p\oplus\C v(p)$, so $d\pi_p\colon V_p\to T_p S$ is a complex linear isomorphism.

Take any $V$-disc $\D\xrightarrow{\varphi}M$.
This $V$-disc projects by $\pi$ to a holomorphic disc \(\D\xrightarrow{\bar\varphi}S\) by \(\bar\varphi:=\pi\circ\varphi\).
By the chain rule, $d\pi$ maps the velocity of \(\varphi\) to the velocity of \(\bar\varphi\):
\[
\bar\varphi'(0)=(\pi\circ\varphi)'(0)=d\pi_p(\varphi'(0)).
\]

Conversely, by Proposition~\vref{Prop:V-lifts} we can lift any disc \(\bar\varphi\) in $S$, say with \(\bar\varphi(0)=\bar p\), to a $V$-disc \(\varphi\) in \(M\) with \(\varphi(0)=p\), so that \(\bar\varphi:=\pi\circ\varphi\).
Again, by the chain rule,
\[
\bar\varphi'(0)=(\pi\circ\varphi)'(0)=d\pi_p(\varphi'(0)).
\]
So the velocity of the lift maps by the isomorphism \(d\pi_p\) to the velocity of the original disc.
By uniqueness of the lift, this inverts the projection.
So if we fix the point $p$, $d\pi_p$ is a bijection of velocities.

The pseudometrics agree because they are defined by the velocities, as their Minkowski functionals:
\[
\kappa_S(\bar p;\bar v)=\kappa_V(p;v).
\]

Paths in $M$, tangent to $V$, from $p$ to $q$ project to paths from $\bar p$ to $\bar q$.
Take such a path $\gamma(t)$, $t_0\le t\le t_1$.
Let $\bar\gamma:=\pi\circ\gamma$:
\begin{align*}
\ell_S(\bar\gamma)
&=
\int_{t_0}^{t_1}
\kappa_S(\bar\gamma(t);\bar\gamma'(t)),
\\
&=
\int_{t_0}^{t_1}
\kappa_S((\pi\circ\gamma)(t);(\pi\circ\gamma)'(t)),
\\
&=
\int_{t_0}^{t_1}
\kappa_S((\pi\circ\gamma)(t);d\pi_{\gamma(t)}\gamma'(t)),
\\
&=
\int_{t_0}^{t_1}
\kappa_V(\gamma(t);\gamma'(t)),
\\
&=
\ell_V(\gamma).
\end{align*}
Every rectifiable path in $S$ from $\bar p$ to $\bar q$ lifts up to $M$ to a unique $V$-tangent path starting at $p$, but the lifted path might not reach $q$.
The infimum path length in $S$ is infimum over a possibly larger set of path lengths:
\[
k_S(\bar p,\bar q)\leq k_V(p,q).
\]
Similarly for the chain pseudodistances.

From here the proof is as in Lemma~\vref{V.hyp.to.hyp}: if $(M,V)$ is contact-hyperbolic and $k_S(\bar p,\bar q)=0$ then there are arbitrarily $V$-short paths in $M$ between $p$ and $\pi^{-1}\bar q$.
By Proposition~\vref{prop:same-top}, $k_V$ induces the manifold topology on $M$.
So every open set about $p$ contains points of $\pi^{-1}\bar q$.
But $\pi(p)=\bar p\ne \bar q$, contradicting the continuity of $\pi$.
\end{proof}
An elementary result on discs in Kobayashi hyperbolic manifolds, controlling how rapidly discs stretch distances as measured in local coordinates:
\begin{lemma}\label{lemma:control.discs}
Suppose that $S$ is a Kobayashi hyperbolic complex manifold.
Every point $s_0\in S$ lies in an open set $U$, identified in some holomorphic chart with the unit ball $\B$, so that $s_0$ becomes the origin.
Pick any \(r,R\) with \(0<r<R<1\).
Let $\B_r:=\{(z,w)\in \C^N\times \C^N: \|(z,w)\|<r\}$. 
For some $\rho>0$, every holomorphic disc taking the origin to a point of $\B_r$ takes the disk of radius $\rho$ to $\B_R$.
\end{lemma}
\begin{proof}
If not, there is a sequence of discs $\{\varphi_n\}$ with $\varphi_n(0)\in\B_r$ and a sequence $\{\zeta_n\}\subset\D$ converging to $0$ such that $\varphi_n(\zeta_n)\in S\setminus\bar\B_R$. Thus,
\[
k_S(\bar\B_r, S\setminus\bar\B_R)\leq k_S(\varphi_n(0), \varphi_n(\zeta_n))\leq k_\D(0,\zeta_n)\to 0.
\]
But $S$ is Kobayashi hyperbolic, a contradiction.
\end{proof}

\begin{theorem}\label{Thm:Reeb-principal-bdle}
A contact symplectic lift \((M,V,v,S,\omega)\) is contact-hyperbolic if and only if $S$ is Kobayashi hyperbolic.
\end{theorem}
\begin{proof}%
By lemma~\vref{V.hyp.to.hyp}, if \((M,V)\) is contact-hyperbolic then $S$ is Kobayashi hyperbolic.
We need only prove the other direction, so assume that $S$ is Kobayashi hyperbolic.

Fix two points $p,q\in M$.
Denote their projections by $\bar p,\bar q\in S$.
We need precisely to prove that $k_V(p,q)>0$, in other words to separate $p$ from $q$.
So assume that $k_V(p,q)=0$.
Our plan:
\begin{itemize}
\item 
Prove that $\bar p=\bar q\in S$.
\item
Localise the problem in $S$ near \(\bar p\) by either of two arguments: 
\begin{enumerate}
\item
Geometrically: relate path lengths in $M$ to those in $S$, using lemma~\vref{lemma:kV.eq.kS}. Relate path lengths in $S$ to Euclidean path lengths in coordinates on $S$, using continuity of the Kobayashi metric.
Estimate Euclidean path lengths in coordinates on $M$ from those on $S$ using quadrature.
\item
Complex analytically: use lemma~\vref{lemma:control.discs} to ensure that all $V$-discs mapping the origin to $p$ project to $S$ to stay near $\bar p$ on a fixed subdisc.
Use a Schwarz--Pick estimate on the derivatives of the disc to ensure that the pseudodistance on $M$ is approximated to a bounded error factor by the pseudodistance on an open subset of $M$, with small projection to $S$.
Controlling the velocities of $V$-discs, we control the pseudometric, hence $V$-lengths, hence the pseudodistance. So short enough paths between $p$ and $q$ stay in a Kobayashi hyperbolic open set in $M$.
\end{enumerate}
\end{itemize}
\begin{ProofSteps}
\ProofStep{Prove that $\bar p=\bar q$}
By lemma~\vref{lemma:kV.eq.kS}, \(k_S(\bar p,\bar q)\le k_V(p,q)\) for any points \(p,q\in M\) and their images \(\bar p,\bar q\in S\).
As $S$ is Kobayashi hyperbolic, \(\bar p=\bar q\).
Connect $p$ and $q$ by $V$-paths of arbitrarily small $V$-length.
These project to paths in $S$ of equal length, by lemma~\vref{lemma:kV.eq.kS}.
\ProofStep{Localise on the base $S$}
Pick an open set $U$ around $\bar p$.
The $k_S$-balls in $S$ induce the manifold topology on $S$.
So short enough paths in $S$ starting at $\bar p$ stay inside $U$.
Projection preserves path length by lemma~\vref{lemma:kV.eq.kS}.
So all $V$-paths of sufficiently small $V$-length stay in the preimage of $U$.
\ProofStep{Introduce coordinates}
Pick $U$ to be a coordinate ball \(U=\B_\varepsilon\) centered at \(\bar p\), in contact symplectic coordinates on \(S\) as in Lemma~\ref{Lem:contact-sympl-coordinates}.
(We might not be able to choose \(U\) to be the unit ball, without changing the symplectic coordinates, since \(\int_U \omega\wedge\bar\omega\) is an invariant.)
\ProofStep{Control vertical motion geometrically}
We have localised on \(S\): our short paths lie in \(U\subseteq S\).
Next we need to control the ``vertical drift'' in fibers of $M\to S$.
By Kobayashi hyperbolicity of $S$, we can bound $k_S$ from below, throughout any closed ball $\bar \B'\subseteq U$, by the Euclidean metric of the coordinates: $k_S(x;v)\ge\alpha|v|$, $x\in U$ \cite[p.~133, Theorem~2]{Ro71}. 
Replace \(U\) by \(\B'\), so this bound holds in \(U\).

Any short enough path $\gamma(t)$, \(t_0\le t\le t_1\) from $p$ to $q$ stays in $U\times\C$.
So we write it in coordinates $\gamma(t)=(z(t),w(t),y(t))$.
Its projection $\bar\gamma(t):=\pi(\gamma(t))\in S$ is given by $\bar\gamma(t)=(z(t),w(t))\in U=\B$, so \(|w(t)|<\varepsilon\) and \(\bar\gamma\) has Euclidean length at most \(\ell_S(\bar\gamma)/\alpha\).

By lemma~\vref{lemma:kV.eq.kS},
\begin{align*}
\ell_V(\gamma)
&=\ell_S(\bar\gamma),
\\
&=\int_{t_0}^{t_1} \kappa_S(z(t),w(t);z'(t),w'(t)),
\\
&\ge\alpha\int_{t_0}^{t_1}|(z',w')|,
\\
&\ge\alpha\int|z'|.
\end{align*}
So \(|w|<\varepsilon\) and \(\int|z'|\le \ell_V(\gamma)/\alpha\).
Because $\gamma$ is $V$-tangent,
\[
y'=\sum_j w_jz'_j.
\]
Bound the Euclidean distance in coordinates between the end points:
\begin{align*}
|y(t_1)-y(t_0)|
&=
\left|\int_{t_0}^{t_1} \sum_j w_jz'_j\right|,
\\
&\le
\operatorname{max}|w|\int_{t_0}^{t_1}|z'|,
\\
&<
\frac{\varepsilon\ell_V(\gamma)}{\alpha}.
\end{align*}
Take \(\ell_V(\gamma)\to 0\): \(y(t_1)=yt(t_0)\) so \(p=q\).
So $(M,V)$ is contact-hyperbolic.
\ProofStep{Localise the vertical motion complex-analytically}%
We provide an alternate proof of localisation in $M$, using tools of complex analysis in place of the geometry above.
To do this, we need to control discs more precisely in $S$ first 
and then control their lifts in $M$.
Let $\epsilon>0$ and let $\gamma_\epsilon$ be a path from \(p\) to \(q\) 
such that 
\begin{equation}\label{eq-lungeps}
\ell_V(\gamma_\epsilon)<k_V(p,q)+\epsilon=\epsilon.
\end{equation}
%
\ProofStep{Controlling discs in $S$}
By lemma~\vref{lemma:control.discs}, there is a ball \(\B\) in $U$, centered at the origin, and constants $r,\rho>0$, so that every holomorphic disc whose center maps to some smaller ball $\B_r$ maps $\D_{\rho}$ into $\B_{2r}$.
This estimate gives precise control on how discs in $S$ stretch, near the origin of each disc.
We can assume that $U=\B$.
\ProofStep{Bounding the vertical component of $V$-discs}
Let $Z=\B\times \C$. 
We denote any $V$-disc as $\varphi(\zeta)=(z(\zeta),w(\zeta),y(\zeta))$.
We claim that, for any $r\in(r,1)$, there is a constant $T>|y_0|$, so that every $V$-disc $\varphi$ with $\pi(\varphi(0))\in\B_r$ and with $|y(0)|\le|y_0|$ satisfies
\[
|y(r_1\zeta)|<T \text{ for all } \zeta\in\D.
\]
To see this, as $\varphi$ is a $V$-disc,
\[
y'(\zeta)=\sum_{j=1}^N w_j(\zeta)z_j'(\zeta).
\]
Integrating
\[
y(\zeta)=\sum_{j=1}^N\int_0^{\zeta}w_j(\tau)z_j'(\tau)d\tau.
\]
\ProofStep{Approximating $V$-discs by $V$-discs with bounded projection}
Let $Z=\B\times \C$. 
For every $p\in M$ such that $\bar p=\pi(p)\in \B_r$ and for every $X\in V_p$, take a $V$-disc $\varphi$ such that $\varphi(0)=p$ and $\varphi'(0)=\lambda X$.
Let $\tilde\varphi(\zeta):=\varphi(\rho\zeta)$ for $\zeta\in \D$, a $V$-disc whose projection to $S$ is contained in $\bar\B_{2r}$, so $\tilde\varphi$ is a $V$-disc in $Z$:
\begin{equation}\label{eq:firstlocZ}
\kappa_{Z,V}(p;X)\leq \frac{1}{\rho}\kappa_V(p;X).
\end{equation}
Here $\kappa_{Z,V}$ is the pseudometric of $\left.V\right|_Z$.
\ProofStep{Controlling vertical motion of $V$-discs with bounded projection}
On $Z$, $V$ is the standard contact structure $\xi_0:=dy-\sum_{j=1}^N w_j dz_j$. We can assume that $p=(0,0,0)$ and $q=(0,0,y_0)$ for some $y_0\in\C\setminus\{0\}$. 

Denote each $V$-disc $\varphi$ in $M$ by
\[
\varphi(\zeta)=(z_1(\zeta),  \ldots, z_N(\zeta), w_1(\zeta),\ldots, w_N(\zeta), y(\zeta))\in \C^N\times \C^N\times \C.
\]
Fix $r_1\in (r,1)$. 
Consider all $V$-discs \(\varphi\) so that $\pi(\varphi(0))\in \B_r$ and $|y(0)|\leq |y_0|$.
We claim that there is $T>|y_0|$ so that all these $V$-discs satisfy
\begin{equation}\label{eq:stima-vyT}
|y(r_1\zeta)|< T \quad \hbox{for all $\zeta\in \D$}.
\end{equation}
The contact condition $i_{\varphi'}\xi=0$ is precisely
\[
y'(\zeta)=\sum_{j=1}^N w_j(\zeta)z_j'(\zeta).
\]
This integrates to
\begin{equation}\label{eq:vary-ball}
y(\zeta)=y(0)+\sum_{j=1}^N \int_0^\zeta w_j(\tau)z_j'(\tau)d\tau.
\end{equation}
Each point $\pi(\varphi(\zeta))=(z(\zeta),w(\zeta))$ lies in $\B_{2r}$, for $|\zeta|<r$.
As $\pi(\varphi(\D))\subset\B$, $|y(0)|\leq |y_0|$ and $|z_j(\tau)|<2r$ for $|\tau|<r$, $j=1,\dots,N$ , for $j=1,\ldots, N$. Also, for $j\in \{1,\ldots, N\}$, $z_j$ is a holomorphic self-map of the unit disc, hence for all $\zeta\in \D$ such that $|\zeta|\leq r_1$
\[
k_\D(z_j(\zeta),z_j(0))\leq k_\D(0, \zeta)\leq k_\D(0,r_1).
\]
Since $z_j(0)\in \{\zeta\in\D: |\zeta|\leq r\}$, and $k_\D$ is complete, there exists $c_0\in (0,1)$ such that $|z_j(\zeta)|\leq c_0$ for all $\zeta\in \D$ such that $|\zeta|\leq r_1$.

By the classical Schwarz--Pick lemma (see, \eg, \cite[Corollary~1.1.4]{AbateTaut}), for all $\zeta\in\D$ such that $|\zeta|\leq r_1$,
\[
|z_j'(\zeta)|\leq \frac{1-|z_j(\zeta)|^2}{1-|\zeta|^2}\leq \frac{1-c_0}{1-r_1}.
\]
By \eqref{eq:vary-ball}, for all $\zeta\in \D$ such that $|\zeta|<r_1$,
\[
|y(\zeta)|\leq |y_0|+N r_1 \frac{1-c_0}{1-r_1}.
\]
Equation \eqref{eq:stima-vyT} follows by taking any $T> |y_0|+N r_1 \frac{1-c_0}{1-r_1}$. 
\ProofStep{Comparing the $V$-metric with the Kobayashi metric on a bounded domain}
Let $W:=\B\times\{y\in\C: |y|<T\}$, a bounded domain in $\C^{2N+1}$, hence Kobayashi hyperbolic; denote it $\kappa_W, k_W$ its Kobayashi metric and distance.
Let 
\[
K:=\{(z,w,y)\in W: \|(z,w)\|<r_1, |y|\leq y_0\}.
\]
Fix $p\in K$ and $X\in V_p$. 
Suppose that $\varphi$ is any $V$-disc in $M$ such that $\varphi(0)=p$ and $\varphi'(0)=\lambda X$.
Let $\tilde\varphi(\zeta)=\varphi(r_1\zeta)$ for $\zeta\in \D$. 
By~\eqref{eq:vary-ball}, $\tilde\varphi(\D)\subset W$. Also, $\tilde\varphi(0)=p$ and $\tilde\varphi'(0)=r_1\lambda X$. By \eqref{eq:firstlocZ}, for every $p\in K$ and $X\in V_p$,
\[
\kappa_W(p;X)\leq \frac{1}{r_1} \kappa_V(p;X)\leq\frac{1}{\rho r_1} \kappa_V(
p;X).
\]
This implies that if $\gamma\colon [0,1]\to K$ is any piecewise smooth curve almost everywhere tangent to $V$, then 
\begin{equation}\label{eq:estima-sotto}
\ell_V(\gamma)\geq \rho r_1 \ell_W(\gamma):=\rho r_1\int_0^1\kappa_W(\gamma(t);\gamma'(t))dt.
\end{equation}
Choose $\epsilon_0>0$ so small that $\pi(\gamma_\epsilon([0,1]))\in \B_r$ for all $\epsilon<\epsilon_0$. The starting and ending points are $\gamma_\epsilon(0)=p=(0,0,0)$ and $\gamma_\epsilon(1)=(0,0,y_0)$. 
Let
\[
t_0:=\sup \{t\in [0,1]: |y_\epsilon(s)|<|y_0|, \forall s\in [0,t]\}.
\] 
Here, $y_\epsilon$ denotes the component of $\gamma_\epsilon$ in the $y$-coordinate.
Then $t_0\in (0,1]$ and $|y_\epsilon(t_0)|=|y_0|$. Therefore, $\gamma_\epsilon([0,t_0])\subset K$.
Let $L:=\{(z,w,y)\in W: |y|=|y_0|\}$.
By \eqref{eq:estima-sotto},
\begin{equation*}\begin{split}
0<k_W((0,0,0), L)& \leq k_W(\gamma_\epsilon(0), \gamma_\epsilon(t_0))\leq \int_0^{t_0}\kappa_W(\gamma_\epsilon(s); \gamma_\epsilon'(s))ds\\ &\leq \int_0^{1}\kappa_W(\gamma_\epsilon(s); \gamma_\epsilon'(s))ds=\ell_W(\gamma_\epsilon)\leq \frac{1}{\rho r_1} \ell_V(\gamma_\epsilon)<\epsilon.
\end{split}
\end{equation*}
As $\epsilon$ is arbitrary, we reach a contradiction. Therefore, $k_V(p,q)>0$ and  $(M,V)$ is contact-hyperbolic.  
\end{ProofSteps}
\end{proof}

Theorem~\vref{thm:S.S.prime} characterises when a scale symplectic map lifts to a contact morphism.
Conversely, contact hyperbolicity of the target contact manifold forces every contact morphism to respect the Reeb vector field:
\begin{corollary}\label{cor:lambda}
Suppose that \((M,V,v)\) and \((M',V',v')\) are Reeb manifolds.
Suppose that \((M',V')\) is contact-hyperbolic and that \(M\) is connected.
Every contact morphism \(M\to M'\) is a Reeb morphism.
\end{corollary}
\begin{proof}
As \((M',V')\) is contact-hyperbolic, \(v'\) has proper flow, \ie \((M',v',V')\) is a proper Reeb manifold.
Take a contact morphism \(M\xrightarrow{\Phi}M'\).
Since the map \(\Phi\) does not increase pseudodistances, \((M,V)\) is also contact-hyperbolic.
Hence \((M,V,v)\) is also a proper Reeb manifold.
So both Reeb manifolds arise from contact symplectic lifts \((M,V,v,S,\omega)\), \((M',V',v',S',\omega')\), unique up to isomorphism.
By Theorem~\vref{Thm:Reeb-principal-bdle}, \(S\) and \(S'\) are Kobayashi hyperbolic.

Any entire curve \(\C\to M\) projects \(\C\to M\to S\) to a constant, since \(S\) is Kobayashi hyperbolic, so is tangent to \(v\).
The same holds in \(M'\).
Since \(\Phi\) maps entire curves to entire curves, it carries Reeb flow lines to Reeb flow lines.
Since \(\Phi\) is an immersion, \(d\Phi_pv(p) =f(p)v'(p')\) for a unique holomorphic function \(M\xrightarrow{f}\C^*\).

Let \(w:=v/f\).
Then $\Phi_*w=v'$.
For any complex time $t$, the time $t$ flow of \(w\) moves \(V\) to a contact structure $e^{tw}V$ on $M$ which, intertwining with the flow of $v'$, has $\Phi_* e^{tw}V=e^tv'\Phi_*V\subseteq e^{tv'}V'=V'$.
But \(\Phi_*^{-1}V'=V\) so \(e^{tw}_*V=V\).
So \(w\) is also a Reeb vector field for \(V\).
By lemma~\vref{lemma:rescaling.Reeb}, \(f\) is constant.
\end{proof}

Theorem~\ref{Thm:Reeb-principal-bdle} allows us to make the following definition:
\begin{definition}
A contact symplectic lift $(M, V, v, S, \omega)$ is  \emph{hyperbolic} if $S$ is Kobayashi hyperbolic, or, equivalently, if $(M,V)$ is contact-hyperbolic.
\end{definition}
\begin{remark}
If \((M,V)\to (M',V')\) is a contact morphism and \((M',V')\) is contact-hyperbolic then \((M,V)\) is contact-hyperbolic.
\end{remark}
\subsection{Complete hyperbolic contact symplectic lifts}
In this section we compare completeness of the contact-hyperbolic distance with completeness of the Kobayashi hyperbolic distance on the symplectic quotient.

\begin{lemma}\label{prop:lift.chains}
Take a contact symplectic lift \((M, V, v, S, \omega)\).
Denote by \(k_S\) the Kobayashi pseudodistance of \(S\).
Take two points \(\bar p,\bar q\in S\) and a point \(p\in M\) so that \(\pi(p)=\bar p\).
Then
\begin{align*}
k_S(\bar p,\bar q)
&=
\inf\set[\tilde k_V(p,q)]{\pi(q)=\bar q},\\
&=
\inf\set[k_V(p,q)]{\pi(q)=\bar q}.
\end{align*}
In particular, the values of these infima are the same for all points \(p\in M\) projecting to \(\bar p\).
If the contact symplectic lift is hyperbolic, then the infima are minima.
\end{lemma}
\begin{proof}
Using finitely many local trivialisations, any path in \(S\) from \(\bar p\) to \(\bar q\) lifts to a unique  $V$-path in \(M\) from \(p\) to some point \(q\) with \(\pi(q)=\bar q\).
By Lemma~\vref{lemma:kV.eq.kS}, these paths have the same length.
Conversely, every $V$-path in \(M\) from \(p\) to some point \(q\) with \(\pi(q)=\bar q\) projects to a path in \(S\) from \(\bar p\) to \(\bar q\), of the same length.
%
Suppose that the contact symplectic lift is hyperbolic.
The distance gives the manifold topology of \(M\), so balls are compact.
As we make paths in \(S\) whose length approaches the distance between the end points, \(q\) stays in a compact set in the fiber.
\end{proof}

Given a Cauchy sequence \(x_1,x_2,\dots\) in a metric space \((X,d)\), an \emph{acceleration} of the sequence is a subsequence \(y_1=x_{i_1},y_2=x_{i_2},\dots\) so that
\[
d(y_j,y_{j+1})+\dots+d(y_{k-1},y_k)\to 0
\]
as \(j,k\to\infty\).
Any subsequence for which the distance between successive points decays faster than exponentially is an acceleration.
In particular, accelerations exist.

The main result of this subsection is:
\begin{theorem}\label{Thm:complete-lift-hyp}
Let $(M, V, v, S, \omega)$ be a hyperbolic contact symplectic lift. Then $(M,V)$ is complete contact-hyperbolic (that is, $k_V$ is a complete distance on $M$) if and only if $S$ is Kobayashi complete hyperbolic.
\end{theorem}
\begin{proof}
Suppose that $(M,V)$ is complete contact-hyperbolic.
Take a Cauchy sequence $\{\bar p_n\}\subset S$.
Replace by an acceleration.
Select a path $\bar\gamma_j$ joining $\bar p_j$ to $\bar p_{j+1}$, parameterised by arc length, whose $V$-length is only slightly larger than the distance.
Concatenating the paths together, we get a path \(\bar\gamma\) passing through all of the points, so that the piece of that path from $p_j$ to $p_k$ has length vanishing when $j,k\to\infty$.
Pick a point \(p_1\in M\) above \(\bar p_1\).
The path \(\bar\gamma\) lifts to a path \(\gamma\), of the same length, starting at \(p_1\).
Successive points \(p_j,p_{j+1}\) have distance no smaller than the length of \(\bar\gamma_j\), so not much larger than distance from \(\bar p_j\) to \(\bar p_{j+1}\).
By the triangle inequality, the sequence upstairs is Cauchy.
By completeness of \(M\), the sequence upstairs converges; by continuity of the projection, the sequence downstairs converges.


Suppose that $S$ is complete Kobayashi hyperbolic.
Take a Cauchy sequence $\{p_n\}\subset M$.
Replace by an acceleration.
Its image $\{\bar p_n\}\subset S$ converges, say to $\bar p$.
Fix a neighborhood $U$ of $\bar p$; the $\bar p_n$ eventually remain in \(U\).
We can assume that \(U\) is the unit ball in contact symplectic coordinates.
After dropping finitely many points from the sequence, we can assume that all \(\bar p_n\) lie in \(U\).
So the \(p_n\) remain in the preimage of $U$ in $M$.
We can pick $U$ to be a ball in the domain of coordinates, as in the proof of Theorem~\vref{Thm:Reeb-principal-bdle}, centered at \(\bar p\), say
 \(p_n=(z_n,w_n,y_n)\) and \(\bar p_n=(z_n,w_n)\to\bar p=(0,0)\).

Pick a $V$-path $\gamma_j$ joining $p_j$ to $p_{j+1}$, parameterised by arc length, whose $V$-length is only slightly larger than the distance.
Concatenating the paths together, we get a path \(\gamma\) so that the portion \(\gamma_{jk}\) of that path from $p_j$ to $p_k$ has length vanishing when $j,k\to\infty$.
Each $V$-path \(\gamma_j\) projects to a path \(\bar\gamma_j\) of the same length in $S$.
As in the proof of Theorem~\vref{Thm:Reeb-principal-bdle}, when we localise in $M$, these short $V$-paths have small differences in $y$-value across their end points.
In the notation of that proof,
\[
|y_k-y_j|\le\frac{\ell_S(\bar\gamma_{jk})}{\alpha}.
\]
Thus, \(\{p_n=(z_n,w_n,y_n)\}\) is Cauchy in the Euclidean metric of the coordinates.

Again we also provide a more complex analytic proof.
We assume that \(U=\B\) and that \(\bar p=(0,0)\).
We first prove a  localisation result of the metric.  Fix $T> 0$ and let  
\[
W_T:=\{(z,w,y)\in \B^{2N}\times \C:  |y|<T\}.
\]
The image of $W_T$ in $M$, through the trivialisation chosen above, is an open subset of $M$. 
Let $\underline{0}=(0,0,0)$.  Since $(M,V)$ is contact-hyperbolic, $k_V$ induces the manifold topology of $M$. Therefore, there exists $R>0$ such that 
\[
B_V(\underline{0}, 4R):=\{q\in Z: k_V(\underline{0}, q)<4R\}\subset\subset W_T.
\]
We claim that there exists $\rho_T>0$ such that for every $V$-disc $\varphi\colon \D\to M$ such that $\varphi(0)\in B_V(\underline{0}, 3R)$ we have $\varphi(\rho_T\D)\subset W_T$. If not, there is a sequence $\{\varphi_n\}$ of $V$-discs such that $\varphi_n(0)\in B_V(\underline{0}, 3R)$ and a sequence $\{\zeta_n\}\subset \D$ converging to $0$ such that $\varphi_n(\zeta_n)\in \partial B_V(\underline{0}, 4R)$. By Lemma~\ref{lem:Schwarz}, we have
\[
k_\D(0, \zeta_n)\geq k_V(\varphi_n(0), \varphi_n(\zeta_n))\geq k_V(\overline{B_V(\underline{0}, 3R)}, \partial B_V(\underline{0}, 4R))=c>0,
\]
a contradiction. Thus, for every $q\in B_V(\underline{0}, 3R)$, every $V$-disc $\varphi\colon \D\to M$ such that $\varphi(0)=q$ has the property that $\varphi(\rho_T\D)\subset W_T$. For every $q\in B_V(\underline{0}, 3R)$ and every $u\in V_q$ we have
\begin{equation}\label{Eq:estim-local2}
\kappa_V(q; u)\geq \rho_T \kappa_{V, W_T}(q;u)\geq \rho_T \kappa_{W_T}(q;u).
\end{equation}
Here, $\kappa_{V, W_T}$ denotes the $V$-infinitesimal metric defined by $\zeta$ on $W_T$ and $\kappa_{W_T}$ denotes the Kobayashi infinitesimal metric of $W_T$.  Since $W_T$ is a bounded domain in $\C^{2N+1}$, it is Kobayashi hyperbolic, therefore (see, \eg, \cite[Prop.~2.3.33]{AbateTaut}) there exists a constant $C_0>0$ such that for every $q\in \overline{B_V(\underline{0}, 2R)}$ and for every $u\in V_q$ we have
\[
\kappa_{W_T}(q;u)\geq C_0 \|u\|.
\]
By \eqref{Eq:estim-local2}, taking $C:=\rho_T C_0$ we have
\begin{equation}\label{Eq:stima-from-below}
\kappa_V(q; u)\geq C \|u\|, \quad \hbox{ for all $q\in \overline{B_V(\underline{0}, 2R)}$ and $u\in V_q$}.
\end{equation}

Fix $\epsilon>0$. We are assuming that $\{\bar p_n\}$ converges to $0$ and $B_V(\underline{0}, R)$ is an open neighborhood of $\underline{0}$.
We can assume that, for every $n\in \N$
\begin{equation}\label{Eq:in-ball-cauchy}
\bar p_n\in B_V(\underline{0}, R).
\end{equation}

Also, $\{p_n\}$ is a Cauchy sequence for $k_V$.
Thus there exists $n_0$ such that $k_V(p_n,p_m)<\epsilon$ for every $n\geq m>n_0$. Let $\gamma_{n,m}\colon [0,1]\to Z$ be a piecewise smooth curve tangent almost everywhere to $V$ such that $\gamma_{n,m}(0)=p_m$, $\gamma_{n,m}(1)=p_n$ and 
\[
\ell_V(\gamma_{m,n})\leq k_V(p_n,p_m)+\frac{1}{n}<\epsilon+\frac{1}{n}. 
\]
Let us write $(z_n,w_n, q_n):=p_n$. The Reeb vector field flow preserves the contact structure and the complex structure on $M$.
So it is an isometry of the pseudometric and the pseudodistance.
Let $Q$ be the time $-q_m$ flow of the Reeb vector field, which, in our local coordinates is $(z,w,y)\mapsto (z,w,y-q_m)$.  So $Q$ is an isometry for $k_V$ and $\kappa_V$ and $\tilde  \gamma_{m,n}:=Q\circ \gamma_{m,n}$ is a piecewise smooth curve tangent almost everywhere to $V$. Therefore, 
\[
\ell_V(\tilde \gamma_{m,n})<\epsilon+\frac{1}{n}. 
\]
By \eqref{Eq:in-ball-cauchy}, $\tilde\gamma_{m,n}(0)=(z_m,w_m,0)=(\pi(p_m),0)\in B_V(\underline{0}, R)$.  
Suppose that there is some $r\in (0,1)$ such that $\tilde\gamma_{m,n}(r)\not\in B_V(\underline{0}, 2R)$.
Then 
\[
\ell_V(\tilde \gamma_{m,n})>k_V(\overline{B_V(\underline{0}, R)}, \partial B_V(\underline{0}, 2R))=:c>0,
\]
reaching a contradiction if $\epsilon$ is sufficiently small and $n_0$ sufficiently large. Therefore, $\tilde\gamma_{m,n}([0,1]))\subset B_V(\underline{0}, 2R)$. 

If $q=(z,w,y)\in \B^{2N}\times \C$, we let $q_y:=y$ (the coordinate in the $y$-component). Thus, by \eqref{Eq:stima-from-below}, taking into account that $\tilde\gamma_{m,n}(1)=Q(p_n)$.
Thus, $(\tilde\gamma_{m,n})_y(0)=0$ and $(\tilde\gamma_{m,n})_y(1)=y_n-y_m$:\begin{equation*}
\begin{split}
|y_n-y_m|&\leq \int_0^1 |(\tilde\gamma_{m,n})_y'(t)|dt\leq \int_0^1 \|\tilde\gamma_{m,n}'(t)\|dt\\&\leq \frac{1}{C} \int_0^1\kappa_V(\tilde\gamma_{m,n}(t)); \tilde\gamma_{m,n}'(t))dt=\frac{1}{C}\ell_V(\tilde \gamma_{m,n})<\frac{1}{C}\left(\epsilon+\frac{1}{n}\right).
\end{split}
\end{equation*}
So $\{y_n\}$ is a Cauchy sequence in $\C$, hence convergent.
\end{proof}
\begin{definition}
A hyperbolic contact symplectic lift $(M, V, v, S, \omega)$ is \emph{complete}  when $(M,V)$ is complete contact-hyperbolic or, equivalently, $S$ is Kobayashi complete hyperbolic.
\end{definition}
\begin{corollary}
The pullback of a complete hyperbolic contact symplectic lift by a holomorphic scale symplectic immersion of complete Kobayashi hyperbolic symplectic manifolds is complete hyperbolic.
\end{corollary}

\section{Contact automorphisms and symplectomorphisms}%
\label{Sec:Lie}
In this section, we relate the automorphism group of a proper Reeb manifold to the automorphism group of its symplectic base manifold.

\begin{lemma}
Suppose that \((S,\omega)\) is a Kobayashi hyperbolic holomorphic symplectic manifold.
Then \(\Aut_{\C\omega}S\) is a Lie group acting smoothly and properly on \(S\).
\end{lemma}
\begin{proof}
Since $S$ is Kobayashi hyperbolic, $\Aut S$ is a Lie group acting smoothly and properly on $S$. 
Clearly, $\Aut_{\C\omega}S$ is a closed subgroup of $\Aut S$.
So it is a Lie subgroup \cite[p.~319 Theorem~9.3.7]{HN2012}.
It is closed, so embedded, so also acts smoothly and properly.
\end{proof}

\begin{theorem}\label{Thm:maps-up-down}
Let $(M, V, v, S, \omega)$ be a hyperbolic contact symplectic lift.
Suppose that $M$ is connected.
Then $M\to S$ is equivariant for a unique morphism of Lie groups $\Aut_V M\to\Aut_{\C\omega} S$.
The kernel of this morphism is a Lie group of real dimension $2$.
It consists precisely of the flow of the Reeb vector field $v$.
We see a short exact sequence of Lie group morphisms
\[
\begin{tikzcd}
1\arrow{r}&\C\arrow{r}&\Aut_V{M}\arrow{r}&\Aut_{\C\omega} S.
\end{tikzcd}
\]
\end{theorem}
\begin{proof}%
By Theorem~\vref{Thm:Riccardo}, \(\Aut_V M\) is a Lie group acting smoothly and properly on $M$.
By Corollary~\vref{cor:lambda}, every automorphism $F\in\Aut_V M$ is a \(\lambda\)-Reeb morphism.
So it is the lift of an associated scale symplectic automorphism \(\hat F\), unique up to Reeb flow.
So \(F\mapsto \hat F\) is a group morphism $\Aut_V M\to\Aut S$, for which $M\to S$ is equivariant.
In the compact-open topology, each element of $\Aut_V M$ acts continuously on $M$ taking Reeb flow lines to one another, so acts continuously on the space of Reeb flow lines: $S=M/\C$.
A continuous Lie group action is smooth, so $\Aut_V M\to \Aut S$ is a smooth Lie group morphism.
By corollary~\vref{cor:autos.id.on.S}, the kernel of this Lie group morphism is the Reeb flow.
\end{proof}
For a contact symplectic lift, this theorem improves the bound on the automorphism group dimension from Proposition~\vref{prop:3bound}:
\begin{corollary}\label{corollary:bound}
Let $(M,V,v,S,\omega)$ be a hyperbolic contact symplectic lift.  
Then
\[
\dimR\Aut_V M\leq2+\dimR\Aut_{\C\omega} S\leq 2+\dimR\Aut S.
\]
In particular, if $S$ has complex dimension $2N$,
\[
\dimR\Aut_V M\leq (2N)^2+4N+2.
\]
\end{corollary}
\begin{proof}
The relation between dimensions follows from Theorem~\vref{Thm:maps-up-down}. 
Since $S$ is Kobayashi hyperbolic we have $\dimR \Aut S\leq (2N)^2+4N$ by Isaev~\cite{Isa07}.
\end{proof}
Not every scale symplectic automorphism lifts to a contact automorphism:
\begin{example}
Continuing Example \ref{non-equiv}, the symplectic manifold $(S=\D^\ast \times \D, \omega = dz \wedge dw)$ has nonisomorphic hyperbolic contact symplectic lifts $(S \times \C, \xi_c = dy - (w - \frac{c}{2\pi iz})dz)$ for $c \in \C$.
Consider the biholomorphism $F \in \Aut_{\C\omega} S$:
\[
F(z,w) = \left(-\frac{1}{z},
z^2w
\right),
\]
which acts non-trivially on $H^1(S, \C)$.
We want to prove that \(F\) has no lift to a contact morphism.

Fix $c \neq 0$.
Let $V_c := \ker \xi_c$ be the associated contact distribution. Any lift $G \in \Aut_{V_c} M$ of $F$ has the form
\[
G(z,w,y) = \left(
-\frac{1}{z}, 
z^2w, 
Y(z,w,y)\right),
\]
and satisfies $G^\ast \xi_c = \lambda \xi_c$, where $F^\ast \omega = \lambda \omega$ so $\lambda = 1$. Explicitly, we have
\[
dy - \left(w - \frac{c}{2\pi iz}\right)dz= \xi_c = G^\ast \xi_c = dY - \left(w + \frac{c}{2\pi i z}\right) dz,
\]
so
\[
d(Y-y)=\frac{c}{\pi i z} dz,
\]
but $\frac{c}{\pi i z} dz$ has no globally defined single valued holomorphic primitive on \(S\).
\end{example}

Theorem~\vref{thm:S.S.prime} characterises which scale symplectic maps lift to contact manifold automorphisms; in particular
\begin{theorem}
Suppose that $(S,\omega)$ is a $C^\infty$-exact holomorphic symplectic manifold and that $H^1(S,\C)=0$.
Then 
\[
\Aut_V M\xrightarrow{\Psi}\Aut_{\C\omega} S
\]
is onto.
In other words, we have a short exact sequence of Lie group morphisms
\[
\begin{tikzcd}
1\arrow{r}&\C\arrow{r}&\Aut_V{M}\arrow{r}&\Aut_{\C\omega} S\arrow{r}&1.
\end{tikzcd}
\]
\end{theorem}
\begin{remark}\label{Rem:trivial-case}
When $\omega$ is holomorphically exact and $H^1(S, \C)=0$, we can write the lift explicitly.
By Corollary~\ref{Cor:unique-contact}, $M=S\times \C$, $\pi\colon S\times \C\to S$ is the projection on the first factor, $v=\partial_y$ and the contact form is
\[
\xi=d y+\pi^\ast\nu,
\]
where $\nu$ is a holomorphic $1$-form on $S$ such that $d\nu=\omega$.  Let $F\in \Aut_{\C\omega} S$ be such that $F^\ast \omega=\lambda \omega$, for some $\lambda\in\C\setminus\{0\}$. 
Since $d(F^\ast\nu-\lambda\nu)=0$ and $H^1(S,\C)=0$, by de~Rham's theorem, $F^\ast\nu-\lambda\nu=dh$ for a smooth function $h$, unique up to a constant. 
Since \(dh\) is a holomorphic \(1\)-form, \(h\) is holomorphic.
Define 
\[
\tilde F\colon S\times\C\to S\times\C, \qquad
\tilde F(\bar p,y):=(F(\bar p), \lambda y-h(\bar p)), \quad (\bar p,y)\in S\times \C.
\]
Check that $\tilde F$ is an automorphism of $S\times\C$ and \(\tilde F^\ast\xi=\lambda \xi\).
So $\tilde F\in \Aut_V M$ and $\Psi(\tilde F)=F$.
\end{remark}

\section{Classification of automorphism groups in low dimension}\label{Sec:class}
In this section we study automorphism groups of hyperbolic Reeb manifolds in low dimension.  
Our goal is to bound the dimension of the automorphism group and to classify the extremal cases.
We have reduced the problem to symplectic geometry on the symplectic quotient.
On the quotient, we apply Isaev’s classification \cite{Isa07} of Kobayashi hyperbolic complex manifolds with large automorphism groups.
We need to consider which subgroups of Isaev's groups can act as scale symplectic automorphisms, hence we start by looking at invariant forms.

\subsection{Invariant forms}
While Isaev's classification gives us a starting point, we still need to preserve symplectic forms up to scale.
We need the following lemma of classical invariant theory to determine which contact, symplectic, or volume forms are compatible with a given group action.

\begin{lemma} \label{lem:torus}
Let $\Omega \subset \C^N$ be a complete Reinhardt domain.
Let
\[
\mathbb T_n:= \{(z_1, \dots z_N) \mapsto (e^{i \theta_1} z_1, \dots, e^{i \theta_n} z_n, z_{n+1},\ldots, z_N) :\ \theta_1,  \dots \theta_n \in \R \} 
\subset \Aut\Omega,
\]
for some $n$ with $1\leq n\leq N$. 
Let \(\mu\) be a nowhere-vanishing holomorphic volume form.
The following are equivalent:
\begin{itemize}
\item
Every element of $\mathbb T_n$ preserves $\mu$ up to a complex scalar factor.
\item
Every element of $\mathbb T_n$ preserves $\mu$.
\item
The volume form \(\mu\) has the form
\[
\mu=f\,dz_1\wedge\dots\wedge dz_N,
\]
where \(f=f(z_{n+1},\dots,z_N)\) is holomorphic and nowhere vanishing.
\end{itemize}
\end{lemma}
\begin{proof}
The torus preserves \(\mu\) up to a complex scalar factor: for each $T\in \mathbb T_n$, there exists $\lambda_T\in \C^\ast$ such that $T^\ast \mu=\lambda_T \mu$. 
Write
\[
\mu:=f(z_1, \dots, z_N) dz_1 \wedge \dots \wedge dz_N,
\]
for some holomorphic nowhere-vanishing function $f$ on \(\Omega\).
We need to prove that $f$ depends only on $(z_{n+1}, \ldots, z_N)$. In particular, if $n=N$ then $f$ is constant.
Set \(z':=(z_2,\dots,z_N)\); we can expand in a Taylor series \cite{Hor90} p.~35:
\begin{equation}\label{Eq:Taylor-exp-volume}
f(z_1,z')=\sum_{j=0}^\infty f_j(z') z_1^j,
\end{equation}
where $f_j$ are holomorphic.
Since $f$ vanishes nowhere, $f_0$ vanishes nowhere. 
For $T_\theta(z_1,z'):=(e^{i\theta}z_1, z')$, $T_\theta^*\mu=\lambda_{\theta}\mu$ expands to \(f(e^{i\theta}z_1,z')e^{i\theta}=\lambda_\theta f(z_1,z')\).
Matching up Taylor coefficients, $\lambda_\theta=e^{i\theta}$ and $f_j(z')=0$ for $j\ge 1$.
So $f(z_1,z')=f_0(z')$; repeat this argument for $z_2,\ldots, z_n$.
\end{proof}
\subsection{Subgroups of semisimple groups}
For our classification result, we need to apply some elementary Lie group theory to automorphism groups:
\begin{theorem}\label{thm:tilted.subgroups}
Let $G$ and $G'$ be simple Lie groups. Every connected codimension-1 subgroup of $G \times G'$ is of the form $N \times G'$ or $G \times N'$ for a connected codimension-1 subgroup $N \subset G$ or $N' \subset G'$. 
\end{theorem}
\begin{proof}
Let $L \subset G \times G'$ be a codimension-1 subgroup.
Let $\LieL$ be its Lie algebra.
Take $\alpha \in T_1 (G \times G')^\ast$ such that $\LieL = \ker \alpha$.
Extend $\alpha$ to a left-invariant form, still denoted by $\alpha$, on $G \times G'$.

The left cosets of $L$ form a foliation of $G \times G'$ with tangent distribution $\ker \alpha$. By the Frobenius theorem, $\alpha \wedge d\alpha = 0$.

Decomposing $T (G \times G') = TG \times TG'$, write $\alpha = \alpha_0 + \alpha_0'$ for two left-invariant forms $\alpha_0$ and $\alpha_0'$ on $G$ and $G'$ respectively. Assuming that $L$ is \emph{not} of the form $N \times G'$ or $G \times N'$, neither $\alpha_0$ nor $\alpha_0'$ vanishes.

Extend  $\{\alpha_0\}$ and $\{\alpha_0'\}$ to bases of left-invariant forms $\{\alpha_0, \alpha_i\}$ and $\{\alpha_0',\alpha_a'\}$ on $G$ and $G'$. By left invariance,
\begin{align*}
d\alpha^0&=c^0_{0j}\alpha_0\wedge\alpha_j+c^0_{ij}\alpha_i\wedge\alpha_j,\\
d\alpha_0'&=c^{\prime 0}_{0b}\alpha_{0}'\wedge\alpha_{b}'+c^{\prime 0}_{ab}\alpha_{a}'\wedge\alpha_b',\\
\end{align*}
for structure constants 
\(c^{\dotIndex}_{\dotIndex\dotIndex}\), \(c^{\prime \dotIndex}_{\dotIndex\dotIndex}\).
Here we use the convention of summation over distinct indices.

This gives
\begin{align*}
0
&=\alpha\wedge d\alpha,
\\
&=
\alpha_0\wedge d\alpha_0
+
\alpha_0'\wedge d\alpha_0
+\alpha_0\wedge d\alpha_0'
+\alpha_0'\wedge d\alpha_0',
\\
&=
c^0_{ij}\alpha_0\wedge\alpha_i\wedge\alpha_j
+c^{\prime 0}_{0b}\alpha_0\wedge\alpha_0'\wedge\alpha_b'
+c^{\prime 0}_{ab}\alpha_0\wedge\alpha_a'\wedge\alpha_b'\\
&\qquad+c^0_{0i}\alpha_0'\wedge\alpha_0\wedge\alpha_i+ 
c^0_{ij}\alpha_0'\wedge\alpha_i\wedge\alpha_j+ 
c^{\prime 0}_{ab}\alpha_0'\wedge\alpha_a'\wedge\alpha_b'. 
\end{align*}

By linear independence of these forms, each term vanishes separately, ensuring that \(d\alpha_0=0\) and $d\alpha_0'=0$.

Let $\{v_0,v_i \}$ be the dual basis of $\{\alpha_0, \alpha_i\}$ and let $\LieN$ be the kernel of $\alpha_0$, spanned by $\{v_i\}$ (excluding $v_0$):
\begin{align*}
\alpha_0([v_i,v_j]) &= v_i (\alpha_0 (v_j)) - v_j (\alpha_0 (v_i)) - d\alpha_0 (v_i, v_j) = 0,
\\
\alpha_0([v_0,v_j]) &= v_0 (\alpha_0 (v_j)) - v_j (\alpha_0 (v_0)) - d\alpha_0 (v_0, v_j) = 0,
\end{align*}
Hence $\LieN$ is a Lie subalgebra and an ideal in the Lie algebra $\LieG$ of $G$. This contradicts the assumption that $G$ is a simple group.
\end{proof}

Fabio Podest\`a provided the authors with a second proof, which gives a stronger result, and which we detail here.
We first require a lemma.
\begin{lemma}\label{lemma:normal.in.closure}
Every connected immersed Lie subgroup of a Lie group is normal in its closure.
\end{lemma}
\begin{proof}
Take a connected immersed Lie subgroup \(H\subset G\).
Under the adjoint \(H\)-action, the Lie algebra of \(H\) is invariant, and is a closed linear subspace in the Lie algebra of \(G\).
Hence the Lie algebra of \(H\) is adjoint invariant under the closure \(\bar H\).
Exponentiating, \(H\) is invariant under conjugation by \(\bar H\).
\end{proof}

Podest\`a then gives an improvement of theorem~\vref{thm:tilted.subgroups}:
\begin{theorem}\label{thm:tilted.subgroups.2}
Let $G$ and $G'$ be simple Lie groups. 
Every connected codimension-1 subgroup of $G \times G'$ is of the form $N \times G'$ or $G \times N'$ for a connected closed codimension-1 subgroup\ $N \subset G$ or $N' \subset G'$. 
\end{theorem}
\begin{proof}
Suppose that \(L\subset G\times G'\) is a codimension-1 connected subgroup.
The closure \(\bar L\) is a Lie subgroup.
Either \(\dimR\bar L = \dimR L\) or \(\bar L\) is the identity component of \(G\times G'\). 
By lemma~\vref{lemma:normal.in.closure}, \(L\) is normal in its closure \(\bar L\).

Suppose that \(\bar L=G\times G'\).
So \(L\) is normal in \(G\times G'\).
The projection of \(L\) to either factor \(G,G'\) is onto or trivial.
If trivial, \(L\) lies in the other factor.
Simple Lie groups have dimension at least \(3\).
So \(L\) has codimension larger than \(1\), a contradiction.
So \(L\) projects onto both factors.
The kernels of the two projections are normal subgroups of \(G',G\) respectively, of dimension equal to \(G',G\) respectively.
So \(G,G'\subseteq L\), hence \(G\times G'\subseteq L\), a contradiction since \(L\) is codimension-\(1\).

So \(L = \bar L\) is a closed subgroup. 
The quotient manifold \(X:=(G\times G')/L\) is 1-dimensional. 
By a classical result of Sophus Lie \cite{Br91} p.~61 Exercise~11, the maximal dimensional algebra of vector fields any \(1\)-dimensional manifold is \(\mathfrak{sl}(2,\R)\).
Indeed, every action of a connected Lie group on a connected \(1\)-dimensional manifold \(X\) factors through the action of the group \(\PSL{2,\R}\) of M\"obius transformations on the real projective line \(\RP{1}\), for an equivariant immersion \(X\to\RP{1}\) with image an invariant open set.
Therefore \(G,G'\) have homomorphisms to \(\PSL{2,\R}\), either onto or trivial since \(\PSL{2,\R}\) is a simple group.
Their images commute, forcing one or the other to be mapped to the identity.
So one of \(G\) or \(G'\) lies inside \(L\); suppose it is \(G\).
For every element \((g,g')\in L\), \((g^{-1},1)\in L\) so \((g,1),(1,g')\in L\), so \(L\) is a product \(L=G\times N\).
\end{proof}

\subsection{Symplectic automorphisms of the bidisc}
We apply Lemma~\vref{thm:tilted.subgroups} to rule out certain group actions:
\begin{corollary}\label{cor:no.5}
No $5$-dimensional Lie group acting faithfully on the bidisc preserves any holomorphic symplectic form, even up to a constant factor.
\end{corollary}
\begin{proof}
Apply Lemma~\vref{thm:tilted.subgroups} with $G = G' = PSL_2\R$.
Up to swapping factors, the group has identity component \(N\times PSL_2\R\), with \(N \subset PSL_2\R \) a connected subgroup. Consider the standard action on the product of upper-half planes $\mathbb{H} \times \mathbb{H}$.
Any holomorphic \(2\)-form \(\omega=f(z,w)dz\wedge dw\) on $\mathbb{H} \times \mathbb{H}$ scaled by translations has
\[
f(z,w+b) = \lambda_b f(z,w),
\]
for all $b \in \R$. 
Differentiate in $b$ to find $f(z,w)=f(z)$.
But then \(w\mapsto -1/w\) is in $PSL_2\R$ and does not preserve \(\omega\) even up to scale.
\end{proof}

\subsection{Three dimensional contact manifolds}
Our main result of this section:
\begin{theorem} \label{thm:classification3}
Let $(M, V, v)$ be a contact-hyperbolic Reeb \(3\)-manifold.
Then
\[
2\leq \dimR \Aut_V M \leq 7. 
\]
If equality holds on the right then the Reeb manifold is Example~\vref{example:big.aut}, up to isomorphism.
\end{theorem}
The bound in Theorem~\ref{thm:classification3} is sharper than the one in Corollary~\vref{corollary:bound}.
\begin{proof}
Since $(M,V)$ is contact-hyperbolic, by Proposition~\vref{prop:3bound}, \((M,V,v)\) is a proper Reeb manifold.
By Lemma~\vref{lemma:Reeb.to.csl}, it arises in a contact symplectic lift $(M,V,v,S,\omega)$.
By Theorem~\vref{Thm:Reeb-principal-bdle}, since $(M,V)$ is contact-hyperbolic, $S$ is Kobayashi hyperbolic.
By Theorem~\vref{thm:C.infty.exact}, \(\omega\) is \(C^\infty\)-exact.
Let $d_S:= \dimR \Aut_{\C\omega} S$ and  $d_M:=\dimR\Aut_V M$.
Since $\Aut_{\C\omega} S \subseteq \Aut S$, $d_S \leq 8$ \cite{Isa07}.
By Corollary~\vref{corollary:bound} and Theorem~\vref{Thm:maps-up-down}, \(d_M\leq2+d_S\).
By Corollary~\vref{corollary:bound} and Theorem~\vref{Thm:maps-up-down}, $d_M=d_S+2$.
So we need to prove that \(d_S\le 5\) and classify the possibilities for \((S,\omega)\) with \(d_S=5\).
Consider possible values of $d_S$:
\begin{ProofSteps}
\ProofStep{$d_S = 8$.} Up to biholomorphism, $\Aut_{\C\omega} S =  \Aut S$ and $S = \B^2$. By Lemma \ref{lem:torus}, $\omega$ is a constant multiple of $dz \wedge dw$; but $\Aut \B^2$ does not preserve this form, even up to scale.
\ProofStep{$d_S = 7$.} No $7$-dimensional group acts properly on a complex surface \cite{Isa08}.
\ProofStep{$d_S = 6$.} By Isaev \cite{Isa08}, up to biholomorphism, $S = \D^2$, and the connected component of the identity in $\Aut_{\C\omega} S$ is $\Aut \D \times \Aut \D$. Once again using Lemma \ref{lem:torus}, we see that the form $\omega$ is a multiple of the standard volume form.
Since $\Aut \D \times \Aut \D$ does not preserve it, we reach another contradiction.
Therefore \(\dimR\Aut_V M\le 7\).
\ProofStep{$d_S = 5$.} So $\dimR \Aut S \geq 5$, so $S= \B^2$ or $\D^2$ \cite{Isa07}. By Corollary~\vref{cor:no.5}, $S\ne\D^2$ so $S=\B^2$.

The group $\Aut_{\C\omega} S$ acts transitively \cite{Isa08}, so by \cite[Lemma A.1]{BIM16} it contains a closed, simply connected, solvable subgroup $H$ acting freely and transitively on $\B^2$. In particular, $\dimR H=4$.
The Lie algebra $\mathfrak{su}(2,1)$ of $\Aut \B^2 = \PSU(2,1)$ admits precisely two solvable Lie subalgebras of dimension 4 up to conjugacy \cite[p.~25 Theorem~6.6, p.~28, Table~11]{DouGraaf25X}.
We now describe the corresponding groups. Both arise inside the 5-dimensional group $G'$ of automorphisms of the Siegel model \(S'=\Ha^2\) of the ball
\[
\Ha^2=\set[(Z,W) \in \C^2]{\Re{Z}>|W|^2},
\] 
fixing the point at infinity, which are precisely \cite{BCD07} those of the form
\[
(Z,W)\mapsto (|a|^2 Z + 2 \bar{s}a W + i c +|s|^2, a W + s) \text { for } a \in \C^\ast, c \in \mathbb{R}, s \in \C.
\]
Let $H'\subseteq G'$ be the subgroup corresponding to $H$ under the Cayley transform.
The subgroup given by $|a|=1$ is not simply connected, so up to conjugation, $H'$ is the subgroup of transformations
\[
(Z,W)\mapsto (a^2 Z + 2 \bar{s}a W + i c +|s|^2, a W + s) \text { for } a>0, c \in \mathbb{R}, s \in \C.
\]
In particular, $H'$ contains all \emph{parabolic transformations} $F_s(Z,W)=(Z + 2 \bar{s}W+|s|^2, W + s)$ for $s \in \C$. 
Some symplectic form $\omega'$ on $S'$ corresponds to $\omega$ on $S$ under the Cayley transform.
Write $\omega' = f(Z,W) dW \wedge dZ$ for some holomorphic function $f$ on $\Ha^2$.
But $\omega'$ is invariant up to scale for this set of transformations:
\[
\lambda(s) f(Z,W) dZ \wedge dW =  f (F_s(Z,W)) dZ \wedge dW, \ \lambda(s) \in \C^\ast.
\] 
Differentiating with respect to $s$ and $\bar{s}$ and evaluating at $s=0$:
\begin{align*}
\partial_s\lambda(0) f &= \partial_W f \\
\partial_{\bar s} \lambda(0) f &= 2W \partial_z f.
\end{align*}
The only solution is $f$ constant; up to scale $\omega' = dZ \wedge dW$.

Since $G'$ is connected, 5-dimensional and preserves $\C\omega'$, it is the connected component of the identity in $\Aut_{\C\omega'}S'$. We show it is the full group.
The Cayley transform \(S=\B^2\to S'=\Ha^2\) is
\[
(z,w)\mapsto (Z,W):=\left(\frac{1+z}{1-z}, \frac{w}{1-z} \right).
\]
It pulls $\omega'=dZ\wedge dW$ back to $\omega:=\frac{2}{(1-z)^3} dz \wedge dw$.
It identifies $G'$ with the closed subgroup $G\subseteq\Aut B^2$ fixing \((1,0)\).
(Every automorphism of $\B^2$ extends homeomorphically through $\partial\B^2$ \cite[Corollary 2.2.2]{AbateTaut}.) 

Suppose that $F \in \Aut\B^2$ satisfies  $F^\ast \tilde \omega = \lambda \tilde \omega$,  $\lambda\in\C\setminus\{0\}$. 
Write $F=(F_1, F_2)$. So
\[
\frac{2}{(1-F_1(z,w))^3} \det (dF_{(z,w)}) dz\wedge dw=F^\ast \tilde \omega=\lambda \tilde \omega=\frac{2\lambda}{(1-z)^3}dz\wedge dw,
\] 
\ie
\begin{equation}\label{Eq:final-form-FB}
|1-F_1(z,w)|^3=\frac{|\det (dF_{(z,w)})|}{|\lambda|} |1-z|^3.
\end{equation}
Since $|\det dF|$ stays bounded near \((1,0)\), the right hand side of \ref{Eq:final-form-FB} tends to $0$ as $(z,w)\to (1,0)$, forcing $F_1(1,0)=1$, hence $F(1,0)=(1,0)$ and $F\in\tilde G$.
\end{ProofSteps}
\end{proof}

\subsection{Three dimensional examples}
Theorem~\ref{thm:classification3} shows that, up to isomorphism, the only contact-hyperbolic Reeb $3$-manifold with $d_M:=\dimR \Aut_V M = 7$ is $\B^2 \times \C$ with contact form $dy + \frac{1}{(1-z)^2} dw$. For $d_M = 6$, the action of $\Aut_{\C\omega} S$ on $S$ need not be transitive.
Suppose that it is transitive.
Then $S$ is either the ball $\B^2$ or the bidisc $\D \times \D$ \cite{OR84}.
We cannot pinpoint a unique symplectic structure on these manifolds as we can when \(d_M=7\). We provide examples with $d_M\le 6$.
\begin{example}
For $S := \B^2$ and $M = S \times \C$ equipped with the standard symplectic and contact structure respectively. The group $\Aut_{\C\omega} S$ coincides with the 4-dimensional Lie group $U(2)$, hence $d_M=6$.
\end{example}
\begin{example}
For $S := \D\times\D$, let $\mathbb H:=\{\zeta\in \C: \Re\zeta>0\}$, $S' := \mathbb H \times \D$, $\omega'(\zeta,w):=d\zeta\wedge dw$, $A:=\{h\in \Aut\mathbb H: h(\zeta)=a\zeta+ib: a>0, b\in \R\}$ the closed subgroup of $\Aut\Ha$ fixing $\infty$ and $G':=\{(f,g)\in \Aut S': f\in A, g(w)=e^{i\theta}w, \theta\in \R\}$.
So $G'\subset \Aut_{\C\omega'}S'$ is a closed 3-dimensional subgroup.
The Cayley transform $C(z)=\frac{1+z}{1-z}$ gives a biholomorphism $\Phi: S\ni (z,w)\mapsto (C(z), w)\in S'$ with
\[
\omega:=\Phi^\ast\omega'=\frac{2}{(1-z)^2}dz\wedge dw.
\]
The Cayley transform identifies $G'$ with the group $G\subset\Aut_{\C\omega}S$ of automorphisms $(z,w)\mapsto (f(z), w)$.
Here $f$ is an automorphism of $\D$ fixing $1$.

The group $\Aut S$ has two connected components \cite{AbateTaut}: the identity component consists of the map $(z,w)\mapsto (f(z), g(w))$ with $f, g \in \Aut \D$.
Compose these with the map $(z,w)\mapsto (w,z)$ to construct the other connected component.

The swap does not preserve $\omega$ even up to scale, so $\Aut_{\C\omega}S\) lies in the identity component.
If $F(z,w):=(f(z),g(w))$ is in this group with $F^*\omega=\lambda\omega$,
\[
\frac{2}{(1-f(z))^2} f'(z)g'(w)=\frac{2\lambda}{(1-z)^2}.
\] 
So $g'$ is constant, $g(w)=e^{i\theta} w$, and $f(1)=1$. So $\Aut_{\C\omega}S=G$ and the unique contact symplectic lift of $(S, \omega)$ is contact-hyperbolic with $d_S=3$ so $d_M=5$.
\end{example}

\begin{example}
For $S := \D \times \D$ and $M = S \times \C$, take the standard symplectic and contact structure. The two-dimensional group $\Aut_{\C\omega} S$ includes the inversion $(z,w) \mapsto (w,z)$ and rotations around both the $z$ and $w$ axes, hence $d_S=2$ and $d_M=4$.
\end{example}

\begin{example}
For $S := \D \times \D$, this time equipped with the symplectic form $(1+z)dz \wedge dw$, and let $M = S \times \C$ with the unique contact-symplectic lift. The group $\Aut_{\C\omega} S$ now allows rotations only around the $w$-axis, hence $d_S=1$ and $d_M=3$.
\end{example}

To obtain an example with $d_M=2$, take any Kobayashi hyperbolic surface with trivial automorphism group and lift it to a contact structure.

\subsection{Five dimensional contact manifolds}
\begin{theorem}
Let $(M,V,v)$ be a contact-hyperbolic proper Reeb contact manifold of complex dimension 5. Then $d_M:=\dimR \Aut_V M \leq 19$.
\end{theorem}
Again, this improves the bound of Corollary~\vref{corollary:bound}.
The classification of these with $d_M=19$ is an open problem.
\begin{proof}
By Theorem~\ref{Thm:Riccardo}, $d_M\le 22$.
As in the proof of Theorem~\vref{thm:classification3}, we study $d_S:=\dimR \Aut_{\C\omega} S$.
By Corollary~\vref{corollary:bound} and Theorem~\vref{Thm:maps-up-down}, $d_M=d_S+2$, so $d_S\leq 20$.
\begin{ProofSteps}
\ProofStep{$d_S = 20,19$} No Lie group of dimension $d_S$ can act properly on a Kobayashi hyperbolic complex manifold of complex dimension 4. \cite{Isa08}.
\ProofStep{$d_S = 18$} By \cite{Isa08}, $S = \B^3 \times \D$ and $\Aut_{\C\omega} S$ has identity component $\Aut \B^3 \times \Aut \D$. By Lemma \ref{lem:torus}, $\omega$ is a multiple of the standard volume form, which is not preserved up to scale by $\Aut \B^3 \times \Aut\D$, a contradiction.
\end{ProofSteps}
So $d_S\le 17$ and so $d_M\le 19$.
\end{proof}


\end{document}